\documentclass[12pt]{article} 

\usepackage{comment}
\usepackage{latexsym, amsthm, amsmath, enumerate}
\usepackage{graphicx, xspace, ifthen, rotating, pict2e}
\usepackage[margin=1in]{geometry}
\usepackage{multirow,multicol}
\usepackage[svgnames]{xcolor}
\usepackage[charter]{mathdesign}
\usepackage{algpseudocode, algorithmicx}
\usepackage{tikz}
\usepackage{tikz-cd}
\usepackage{ytableau}
\usepackage[all,cmtip,color]{xy}

\usepackage{enumitem}
\makeatletter
\newcommand{\mylabel}[2]{#2\def\@currentlabel{#2}\label{#1}}
\makeatother

\usepackage{young}
\ysetshade{Blue!30}
\ysetaltshade{Green!50}
\YSetShade{Green!30}
\YSetAltShade{Purple!45}
\YSETSHADE{Purple!31}
\YSETALTSHADE{Orange!37}

\newlength\circlesize
\allowdisplaybreaks[1]

\usepackage[cmtip,all]{xy}
\newcommand{\longsquiggly}{\xymatrix{{}\ar@{~>}[r]&{}}}

\newcommand{\SSYT}{\mathrm{SSYT}}
\newcommand{\ShST}{\mathrm{ShST}}

\newcommand{\height}{\mathrm{ht}}

\newcommand{\B}{\mathcal{B}}

\newcommand{\wt}{\mathrm{wt}}

\newtheorem{lemma}{Lemma}
\newtheorem{theorem}[lemma]{Theorem}

\newtheorem{proposition}[lemma]{Proposition}

\theoremstyle{definition}
\newtheorem{example}[lemma]{Example}
\newtheorem{definition}[lemma]{Definition}
\newtheorem{remark}[lemma]{Remark}

\numberwithin{equation}{section}
\numberwithin{figure}{section}
\numberwithin{table}{section}
\numberwithin{lemma}{section}

\definecolor{DarkBlue}{rgb}{0, 0.1, 0.55}
\definecolor{DarkRed}{rgb}{0.45, 0, 0}
\newcommand{\defn}[1]{\textbf{#1}}

\newcommand{\east}[1]{\ensuremath{\xrightarrow{\ #1\ }}}
\newcommand{\west}[1]{\ensuremath{\xleftarrow{\ #1\ }}}
\newcommand{\north}[1]{\ensuremath{\big \uparrow \!\text{\raisebox{.1ex}{\scriptsize $#1$}}}}
\newcommand{\south}[1]{\ensuremath{\big \downarrow \!\text{\raisebox{.1ex}{\scriptsize $#1$}}}}

\newcommand{\stepnorth}[2]{\vector(0,1){.92}\put(-.23,.4){\scriptsize$#1$}\put(0,1){#2}}
\newcommand{\stepnorthA}[2]{\vector(0,1){.92}\put(.05,.4){\scriptsize$#1$}\put(0,1){#2}}

\newcommand{\stepsouthshiftW}[2]{\put(-.08,0){\vector(0,-1){.92}\put(-.23,-.6){\scriptsize$#1$}}\put(0,-1){#2}}
\newcommand{\stepeast}[2]{\vector(1,0){.92}\put(.4,-.25){\scriptsize$#1$}\put(1,0){#2}}
\newcommand{\stepeastA}[2]{\vector(1,0){.92}\put(.4,.05){\scriptsize$#1$}\put(1,0){#2}}

\newcommand{\stepwestshiftS}[2]{\put(0,-.08){\vector(-1,0){.92}\put(-.5,-.25){\scriptsize$#1$}}\put(-1,0){#2}}
\newcommand{\std}{\mathrm{std}}
\newcommand{\veps}{\varepsilon}
\newcommand{\vphi}{\varphi}
\newcommand{\vepshat}{\widehat{\varepsilon}}
\newcommand{\vphihat}{\widehat{\varphi}}

\title{Axioms for shifted tableau crystals}

\author{Maria Gillespie\thanks{Partially supported by NSF grant PDRF 1604262.} \\ \texttt{mgillespie@math.ucdavis.edu} 
\and
Jake Levinson\thanks{Partially supported by NSERC grant PDF-502633.} \\
\texttt{jlev@math.uw.edu}
}

\begin{document}
\maketitle

\begin{abstract}
  We give local axioms that uniquely characterize the crystal-like structure on shifted tableaux developed in \cite{GLP}.  These axioms closely resemble those developed by Stembridge \cite{Stembridge} for type A tableau crystals.  This axiomatic characterization gives rise to a new method for proving and understanding Schur $Q$-positive expansions in symmetric function theory, just as the Stembridge axiomatic structure provides for ordinary Schur positivity.
\end{abstract}

\section{Introduction}\label{sec:intro}

  Crystal bases were first introduced by Kashiwara \cite{Kashiwara} in the context of the representation theory of the quantized universal enveloping algebra $U_q(\mathfrak{g})$ of a Lie algebra $\mathfrak{g}$ at $q=0$.  Since then, their connections to tableau combinatorics, symmetric function theory, and other parts of representation theory have made crystal operators and crystal bases the subject of much recent study.  (See \cite{Schilling} for an excellent recent overview of crystal bases.) Crystals have also made appearances in geometry, in particular in the structure of the unipotent geometric crystals defined by Berenstein and Kazhdan (see \cite{Berenstein}, \cite{Berenstein2}).

  In type A, crystal bases are well-understood in terms of semistandard Young tableaux. There are combinatorial operators $E_i$ and $F_i$ that respectively raise and lower the weight of a tableau, by changing an $i+1$ to an $i$ or vice versa. These operators are well-understood explicitly in terms of modifying the reading word of the tableau; see \cite{Kashiwara} or \cite{Schilling} for these definitions.

  Moreover, Stembridge \cite{Stembridge} characterized the intrinsic combinatorial structure of these crystals in terms of local axioms, viewing $\B = \SSYT(\lambda,n)$ as a graph with edges $\xrightarrow{F_i}$, as follows.
\begin{theorem}[Stembridge] \label{thm:stembridge}
Let $G$ be a directed graph with vertices weighted by $\mathbb{Z}^n$ and edges labeled $1, \ldots, n{-}1$. Suppose $G$ satisfies the local axioms (K1-K2), (S1-S3) described in section \ref{sec:background}.

Then each connected component $C$ of $G$ has a unique maximal element $g^*$, with $\lambda = \wt(g^*)$ a partition, and there is a canonical isomorphism $C \cong \SSYT(\lambda,n)$ of weighted, edge-labeled graphs.
\end{theorem}
This result gives a method to show that a symmetric function is Schur-positive: one introduces operators $e_i, f_i$ on the underlying set, satisfying the local axioms. Theorem \ref{thm:stembridge} then implies that the weight generating function is Schur-positive, and enumerating the highest weight elements of the connected components gives the coefficients in the Schur expansion. This method has recently been applied successfully by Morse and Schilling \cite{Morse-Schilling} to certain affine Stanley symmetric functions.

In this paper, we give analogous local axioms for the crystal-like structure introduced in \cite{GLP} on shifted semistandard tableaux. We write $\B = \ShST(\lambda/\mu,n)$ for the set of shifted semistandard tableaux of shifted skew shape $\lambda/\mu$ and entries $\leq n$. The raising and lowering operators defined in \cite{GLP} were introduced to answer geometric questions involving the cohomology of the odd orthogonal Grassmannian $H^*(OG(n,V))$, where $V$ is a $(2n+1)$-dimensional complex vector space with a nondegenerate symmetric form \cite{GLP-geometry}. Combinatorially, these operators are {\bf coplactic}, that is, their action is invariant under shifted jeu de taquin slides, and this property essentially determines their action on \textit{all} shifted skew semistandard tableaux (see Sections \ref{sec:background} and \ref{sec:tableaux} for additional discussion).

Interestingly, although the crystals' combinatorial and enumerative properties describe type B Schubert calculus, the crystal itself resembles a `doubled' type A crystal, with two sets of operators $e_i, f_i$ and $e_i', f_i'$. Accordingly, our axioms resemble a `doubled' form of the Stembridge axioms.

\begin{remark} In type B, there are two related classes of tableaux, `$P$-tableaux' and `$Q$-tableaux', enumerated respectively by the Schur $P$- and $Q$-functions. A crystal structure on $P$-tableaux, corresponding to the representation theory of the quantum queer superalgebra $\mathfrak{q}(n)$, was introduced in \cite{GJKKK}, and the combinatorics and structure of these crystals were further studied in \cite{Assaf}, \cite{GJKKK2}, and \cite{Hiroshima}.  In \cite{ChoiKwon}, Choi and Kwon use the structure to understand Schur $P$-positivity of certain skew Schur functions.

The crystals studied in this paper are on $Q$-tableaux and are nonisomorphic to the $\mathfrak{q}(n)$ crystals. The Schur $P$- and $Q$-functions, first defined by Schur \cite{Schur}, are dual classes of symmetric functions under the standard Hall inner product \cite{Macdonald}, and it would be interesting to understand whether these crystals are also in some sense `dual' to each other.
\end{remark}

Our main theorem is as follows.

\begin{theorem}\label{thm:axioms}
Let $G$ be a directed graph with vertices weighted by $\mathbb{Z}^n_{\geq 0}$ and edges labeled by $1', 1, \ldots, (n{-}1)',n{-}1$. Suppose $G$ satisfies the axioms stated in section \ref{sec:axioms} below.

Then every connected component $C$ of $G$ has a unique maximal element $g^*$, with $\sigma = \mathrm{wt}(g^*)$ a strict partition, and there is a canonical isomorphism $C \cong \ShST(\sigma,n)$.
\end{theorem}

This paper is structured as follows. In Section \ref{sec:background}, we recall the notions of Kashiwara crystals and the Stembridge axioms, and also recall the definitions of the coplactic operators on shifted tableaux defined in \cite{GLP}. In Section \ref{sec:axioms}, we state our axioms. In Section \ref{sec:proof-of-uniqueness} we prove a `uniqueness' statement, that the graphs satisfying our axioms are characterized by their highest-weight elements. In Section \ref{sec:tableaux} we recall the crystal operators on shifted tableaux from \cite{GLP} and show that they obey the given axioms, completing the proof of Theorem \ref{thm:axioms}.

\section{Background}\label{sec:background}

We first recall some basic combinatorial definitions. A \textbf{skew shape} $\lambda/\mu$ is the difference of the Young diagrams of two partitions $\lambda$ and $\mu$, where $\mu_i\le \lambda_i$ for all $i$.   A \textbf{semistandard Young tableau} of skew shape $\lambda/\mu$ is a filling of the Young diagram of $\lambda$ with positive integers in such a way that the entries are weakly increasing across rows and strictly increasing down columns (see Figure \ref{fig:ssyt}).  
\begin{figure}
\begin{center}
\includegraphics{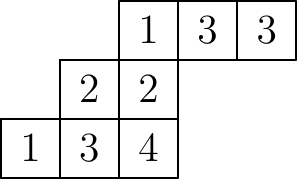}
\end{center}
\caption{\label{fig:ssyt}A semistandard Young tableau of shape $(5,3,3)/(2,1)$.}
\end{figure}

The \textbf{row reading word} of a semistandard Young tableau is obtained by reading its rows from bottom to top.  In Figure \ref{fig:ssyt}, the row reading word is $13422133$.

\subsection{Type A tableau crystals}

Let $\B_{\lambda/\mu} = \SSYT(\lambda/\mu,n)$ be the set of all semistandard Young tableaux of a given skew shape $\lambda/\mu$ (which may be a straight shape if $\mu=\emptyset$) and with entries from $\{1,\ldots,n\}$ for some fixed $n$.  If $T\in \B_{\lambda/\mu}$, let $w$ be its row reading word.
  
  The crystal operators $E_i,F_i:\B_{\lambda/\mu}\to \B_{\lambda/\mu}\cup\{\varnothing\}$ are defined directly in terms of the reading word $w$, as follows. First replace each $i$ in $w$ with a right parentheses and each $i+1$ in $w$ with a left parentheses. For instance, if $w=112212112$ and $i=1$, the sequence of brackets is $))(()())($.  After maximally pairing off the parentheses, $E_i(T)$ is formed by changing the first unpaired $i+1$ to $i$, and $F_i(T)$ is formed by changing the last unpaired $i$ to $i+1$ (they are defined to be $\varnothing$ if the operation is impossible.)
  \[\varnothing\ \xleftarrow{\ E\ }\ \ ))\underline{(()())})\ \ \xleftarrow{\ E\ }\ w = ))\underline{(()())}(\ \ \xrightarrow{\ F\ }\ \ )(\underline{(()())}(\ \ \xrightarrow{\ F\ }\ \ ((\underline{(()())}(\ \ \xrightarrow{\ F\ }\ \varnothing
  \]
  The functions $\varphi_i(T)$ and $\varepsilon_i(T)$ can be defined as the smallest $k$ for which $F_i^k(T)=\varnothing$ or $E_i^k(T)=\varnothing$ respectively, and the weight function $\wt(T)$ is simply the weight vector $(m_j)$ where $m_j$ is the number of $j$'s that occur in $T$.  Then the tuple $(\B_{\lambda/\mu}, E_i,F_i,\varphi_i,\varepsilon_i,\wt)$ gives a type A crystal.  When $\lambda/\mu=\lambda$ is a straight shape, it is precisely the connected crystal corresponding to the irreducible representation of $U_q(\mathfrak{gl}_n)$ indexed by $\lambda$.

\subsection{Type A Stembridge axioms}

Recall that a (finite) {\bf Kashiwara crystal for $\mathrm{GL}_n$} is a set $B$ together with partial operators $e_i,f_i$ on $B$, length functions $\varepsilon_i,\varphi_i:B\to \mathbb{Z}$ for $1 \leq i \leq n-1$, and weight function
$\mathrm{wt}:B\to \mathbb{Z}^n$, such that:
\begin{enumerate}
\item[\mylabel{ax:kashiwara1}{(K1)}]  The operators $e_i, f_i$ are partial inverses, and if $Y = e_i(X)$, then
\[
(\varepsilon_i(Y), \varphi_i(Y))=(\varepsilon_i(X)-1, \varphi_i(X)+1) \ \ \text{ and } \ \
\mathrm{wt}(Y) = \mathrm{wt}(X)+\alpha_i,
\]
where $\alpha_i = (0,\ldots,1,-1,\ldots,0)$ is the weight vector with $1,-1$ in positions $i,i{+}1$.

\item[\mylabel{ax:kashiwara2}{(K2)}] For any $i\in \{1,\ldots,n-1\}$ and any $X\in B$, we have $\varphi_i(X)=\langle \mathrm{wt}(X),\alpha_i\rangle+\varepsilon_i(X)$.
\end{enumerate}

We frequently view $B$ as a directed graph, with edges $\xrightarrow{\ i\ }$ corresponding to $f_i$. For each $i$, we draw the $i$-connected components as `strings':
\[\bullet \xrightarrow{\ i\ } \cdots \xrightarrow{\ i\ } e_i(w) \xrightarrow{\ i\ } w \xrightarrow{\ i\ } f_i(w)\xrightarrow{\ i\ } \cdots \xrightarrow{\ i\ } \bullet\]

We say $B$ is a {\bf Stembridge crystal} if $B$ is a Kashiwara crystal satisfying the following additional axioms.

\begin{enumerate}
\item[\mylabel{ax:Stembridge-ij}{(S1)}] If $|i-j| > 1$ and $f_i(w), f_j(w)$ are defined, then $f_if_j(w)$ and $f_jf_i(w)$ are defined and equal; likewise for $e_i, e_j$.
\item[\mylabel{ax:Stembridge-lengths}{(S2)}] If $f_{i \pm 1}(w) = x$ is defined, then
\[(\varepsilon_i(w) - \varepsilon_i(x),\varphi_i(w) - \varphi_i(x)) = (0,-1) \text{ or } (1,0).\]
\end{enumerate}

\noindent Axiom (S1) states that $i$- and $j$-edges form `commuting squares':
\[\includegraphics{square}\]
It remains to describe the interactions between $i$- and $i{\pm}1$-edges. Axiom (S2) states that the $i$-string lengths can change in precisely two ways when following an edge $w \xrightarrow{i \pm 1} x$:
\begin{center}
\begin{tabular}{l}
\includegraphics[scale=0.8]{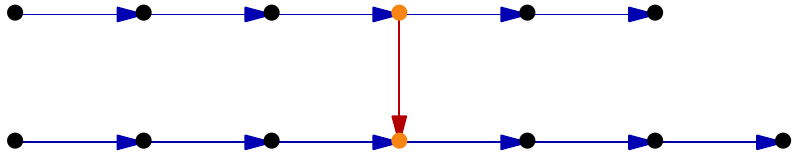} \hspace{1cm} \raisebox{0.5cm}{OR} \hspace{1cm}
\includegraphics[scale=0.8]{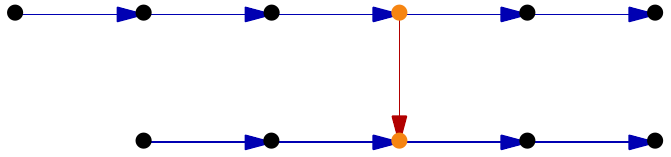}
\end{tabular}
\end{center}
That is, the $i$-string either lengthens by 1 at the low-weight end, or shortens by 1 at the high-weight end. These two possibilities control the explicit relation satisfied by $F_i$ and $F_{i \pm 1}$ in this case.

\begin{enumerate}
\item[\mylabel{ax:Stembridge-merge}{(S3)}] Suppose $f_i(w) = x$ and $f_j(w) = y$ are both defined, where $|i-j|=1$. Let
\[\Delta = (\varepsilon_j(w) - \varepsilon_j(x), \varepsilon_i(w) - \varepsilon_i(y))\]
So $\Delta = (1,1), (1,0), (0,1)$ or $(0,0)$. If $\Delta \ne (0,0)$, then $f_if_j(w) = f_jf_i(w) \ne \varnothing$. Otherwise, $f_if_j^2f_i(w) = f_jf_i^2f_j(w) \ne \varnothing$.
\end{enumerate}


Finally, there is a `dual' axiom to (S3):
\begin{enumerate}
\item[\mylabel{ax:Stembridge-dual-merge}{(S3$^\ast$)}] Suppose $e_i(w) = x$ and $e_j(w) = y$ are both defined, where $|i-j|=1$. Let
\[\Delta = (\varphi_j(w) - \varphi_j(x), \varphi_i(w) - \varphi_i(y))\]
If $\Delta \ne (0,0)$, then $e_ie_j(w) = e_je_i(w) \ne \varnothing$. Otherwise, $e_ie_j^2e_i(w) = e_je_i^2e_j(w) \ne \varnothing$.
\end{enumerate}

\begin{remark}
  The phrasings of Axioms \ref{ax:Stembridge-ij} through \ref{ax:Stembridge-dual-merge} above are slightly different from, but logically equivalent to, Stembridge's original statements in \cite{Stembridge}.  We organize them in this manner in order to draw parallels with our own axiomatic characterization below.
\end{remark}

The proof of Stembridge's Theorem \ref{thm:stembridge} proceeds as follows. Assuming $B$ satisfies these axioms, one shows first that $B$ has a unique element of highest weight. The structure of $B$ is then uniquely determined (working inductively from this maximal element).  The tableau crystals described above satisfy these axioms, and so the tableau crystal $\mathcal{B}_\lambda$ is the unique type A crystal having highest weight $\lambda$.

\subsection{Shifted tableau crystals}

The type A crystal operators are known to be compatible with \textit{jeu de taquin} slides. In fact, the operators are the \emph{unique} operators on skew semistandard tableaux that (i) raise and lower the weight of the tableau appropriately, and (ii) are \textbf{coplactic}, that is, they commute with all sequences of jeu de taquin slides. In \cite{GLP}, similar conditions were used to define raising and lowering operators on \textit{shifted tableaux}. 

We briefly recall the tableaux and operators of interest, but postpone detailed definitions until Section \ref{sec:tableaux}. We use the conventions of \cite{Sagan,Worley} for shifted skew shapes and shifted $Q$-tableaux.  In particular, a \textbf{strict partition} is a tuple $\lambda=(\lambda_1 > \ldots > \lambda_k)$, and its \textbf{shifted Young diagram} is the partial grid of squares in which the $i$-th row contains $\lambda_i$ boxes and is shifted to the right $i$ steps.  A \textbf{(shifted) skew shape} is a difference $\lambda/\mu$ of two partition diagrams $\lambda$ and $\mu$ (if $\mu$ is contained in $\lambda$). 

We write $\ShST(\lambda/\mu,n)$ for the set of skew, shifted \textbf{semistandard tableaux} of shape $\lambda/\mu$ on the alphabet $\{1'<1<2'<2<\cdots<n'<n \}$.  Here, a tableau is semistandard if the unprimed entries repeat only in rows, primed entries repeat only in columns, and for each $i$, the southwesternmost $i$ or $i'$ is unprimed. The {\bf weight} of such a tableau is the vector $(n_1, n_2, \ldots)$, where $n_i$ is the total number of entries $i$ and $i'$.
 An example is as follows:

  \begin{center}
$T =$ \raisebox{-.5\height}{\includegraphics{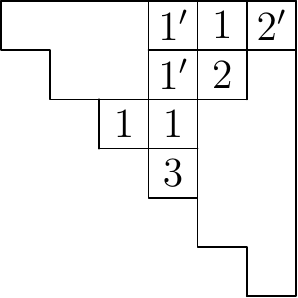}} \qquad $\wt(T) = (4,2,1)$.
  \end{center}

The \textbf{reading word} of a semistandard shifted tableau, as in the unshifted case, is the word formed by concatenating the rows from bottom to top (above, the reading word is $3111'21'12'$).  The reading word arranges the boxes in \textbf{reading order}.

For a rectified shape $\lambda = (a,b)$ with at most two rows, we arrange the tableaux of $\ShST(\lambda,2)$ by weight and define lowering operators as indicated in Figure \ref{fig:two-row-intro}. We then define $F_1, F_1'$ as the \emph{unique} coplactic operators (with respect to shifted jeu de taquin) on shifted skew tableaux that act as indicated on rectified tableaux. We define $F_i, F_i'$ analogously by treating the letters $i',i,i{+}1',i{+}1$ as $1',1,2',2$, and we define $E_i, E_i'$ as the partial inverse operators to $F_i, F_i'$.

\begin{figure}[h]
\begin{center}
 \includegraphics[width=.85\linewidth]{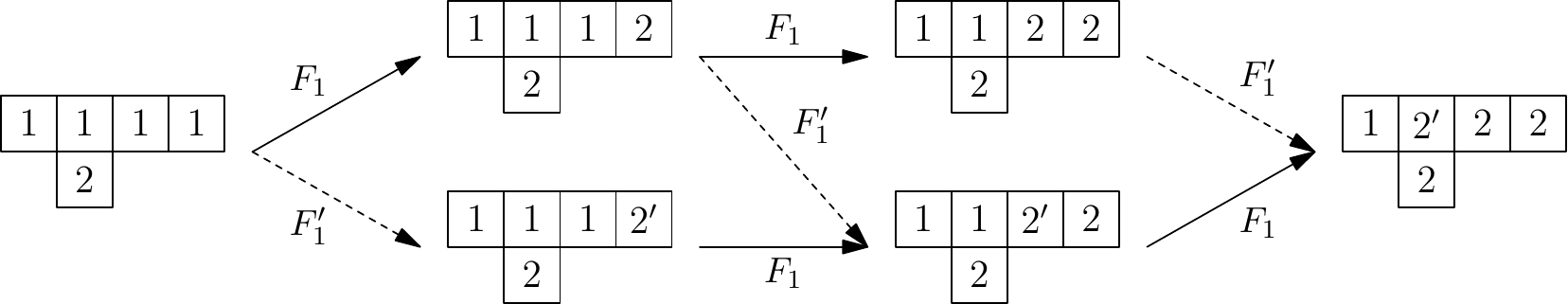} \vspace{1cm}
 
 \includegraphics[width=.85\linewidth]{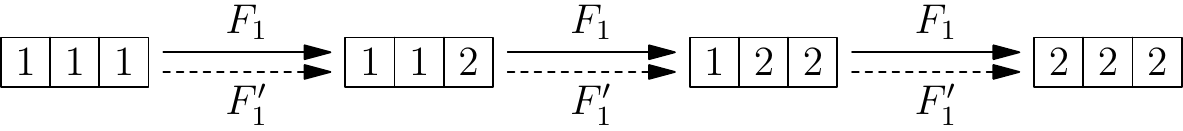}
\end{center}
 \caption{\label{fig:two-row-intro}The crystals on straight shape shifted tableaux with entries $1,1',2,2'$ are either `separated' or `collapsed' strings.  Reversing the arrows gives the partial inverses $E_1$ and $E_1'$.} 
\end{figure}

The main result of \cite{GLP} is to give a direct description of $F_i'$ and $F_i$, depending only on the reading word, on general skew shifted tableaux. We will recall these descriptions in Section \ref{sec:tableaux}, when we show that $F_i, F_i'$ satisfy the axioms stated in Section \ref{sec:axioms}.

\subsubsection{The involution $\eta$}
\label{subsec:eta}
There is a coplactic weight-reversing involution $\eta$ on $\ShST(\alpha,n)$ which reverses arrows in the coplactic structure:
\[\eta \circ F_i \circ \eta = E_{n-i} \text{ and } \eta \circ F_i' \circ \eta = E_{n-i}'\]
for each $i=1, \ldots, n-1$. If $\alpha$ is a straight shape, $\eta$ turns out to be the well-known evacuation involution defined by Sch\"{u}tzenberger \cite{Schutzenberger}, and in general it is the unique coplactic extension of evacuation.  

In our axioms below, we do not assume the graph $G$ has such an involution.  Instead we follow Stembridge's approach and state ``dualized'' forms of each axiom, obtained by reversing all arrows and indices. Thus a relation between $F_i$ and $F'_{i+1}$ has a companion relation between $E_{i+1}$ and $E'_i$ for all $i$.

\section{Axioms for shifted tableau crystals}\label{sec:axioms}

We now state a collection of local axioms governing the structure of the shifted tableau crystals, in the same vein as the Stembridge axioms for ordinary crystals. In this section, we will view a {\bf doubled crystal} as a directed graph $G$, with vertices weighted by $\mathbb{Z}^n$ and edges labeled $1', 1, \ldots, n{-}1', n{-}1$.  

\subsection{Basic structure axioms}

Throughout, $i\in \{1,2,\ldots,n-1\}$ is any valid edge label index, with $i'$ denoting the corresponding primed edge label. In all diagrams, an omitted edge label is either unprimed (if the edge is solid) or primed (if the edge is dashed).  

\begin{itemize}
\item[\mylabel{ax:basic1}{(B1)}] Each $\{i,i'\}$-connected component has one of the following two shapes (of any valid length):
\begin{center}
\raisebox{-.5\height}{\includegraphics[scale=1]{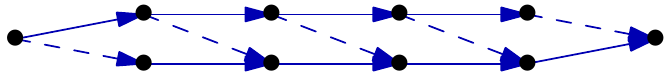}} \ OR \ 
\raisebox{-.5\height}{\includegraphics[scale=1]{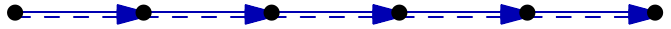}}
\end{center}
We call the type on the left, consisting of two solid (unprimed) chains of equal length connected by dashed (primed) edges, a \textbf{separated string}.  We call the type on the right a \textbf{collapsed string}.  The shortest possible separated string consists of two vertices, connected by only an $i'$ arrow.  A singleton node is considered a collapsed string.

Moreover, $v$ is the top (resp. bottom) of a collapsed string if and only if $\mathrm{wt}_{i{+}1}(v) = 0$ (resp. $\mathrm{wt}_i(v) = 0$). 

\item[\mylabel{ax:basic2}{(B2)}] If $|i-j| > 1$, all edges and reverse-edges (either primed or unprimed) commute. That is, given two edges $w \xrightarrow{a} x, w \xrightarrow{b} y$, there is a vertex $z$ with edges $x \xrightarrow{b} z, y \xrightarrow{a} z$, and conversely.  
\end{itemize}

From \ref{ax:basic1}, we define $f_i(v)$ and $f_i'(v)$, for $v \in G$, by following the unique $i$ or $i'$ edge from $v$, if it exists, and $\varnothing$ otherwise.  We define $e_i, e_i'$ as the (partial) inverse operations. 

We also define statistics $\varepsilon_i, \varphi_i$ on $G$, and require that they satisfy \ref{ax:kashiwara1} and \ref{ax:kashiwara2} independently using both $e_i, f_i$ and $e_i', f_i'$.  

\begin{itemize}
  \item[\mylabel{ax:Kashiwara-B}{(K)}] The statements \ref{ax:kashiwara1} and \ref{ax:kashiwara2} hold for $e_i,f_i$ along with the statistics $\varepsilon_i,\varphi_i$.  They also hold for $e_i',f_i'$ along with the same statistics $\varepsilon_i,\varphi_i$.
\end{itemize}

In what follows, we will further require extra ``primed'' and ``unprimed'' statistics $\varepsilon_i',\varphi_i',\widehat{\varepsilon}_i,\widehat{\varepsilon}_i$ in our setting, as illustrated in Figure \ref{fig:stats}.

\begin{definition}\label{def:stats}
We let $\varepsilon_i(v)$ and $\varphi_i(v)$ be the \textit{total} distance from $v$ to the top and bottom of its $\{i,i'\}$-connected component. We similarly define $\varepsilon_i'(v), \varphi_i'(v)$ (counting primed edges only) and $\widehat{\varepsilon}_i(v), \widehat{\varphi}_i(v)$ (counting unprimed edges only). 
\end{definition}

\begin{figure}[b]
 \begin{center}
  \includegraphics{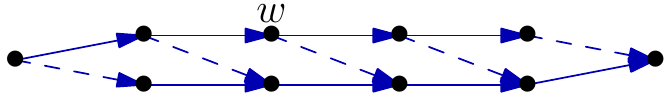}
 \end{center}
 \caption{\label{fig:stats} In the above diagram, we have $\varepsilon(w)=2$, $\widehat{\varepsilon}(w)=2$, $\varepsilon'(w)=0$, $\varphi(w)=3$, $\widehat{\varphi}(w)=2$, and $\varphi'(w)=1$.}
\end{figure}

Another axiom we need is concerns the lengths $\varepsilon_i, \varphi_i$ of strings connected by an $(i \pm 1)$-edge.

\begin{itemize}
\item[\mylabel{ax:length}{(B3)}]Suppose an arrow $z \xrightarrow{i\pm 1} w$ or $z \xrightarrow{i \pm 1'}w$ exists. Then for the $i$-strings through $w$ and through $z$, one of the following two cases holds:
\begin{itemize}
\item[(i)] $\varepsilon_i(w) = \varepsilon_i(z)$ and $\varphi_i(w) = \varphi_i(z) + 1$, or
\item[(ii)] $\varepsilon_i(w) = \varepsilon_i(z)-1$ and $\varphi_i(w) = \varphi_i(z)$.
\end{itemize}
\end{itemize}
Note that Axiom \ref{ax:length} does not specify the changes to $\widehat{\varepsilon}, \varepsilon'$ and $\widehat{\varphi}, \varphi'$, only to $\varepsilon$ and $\varphi$. 
 A precise depiction of this axiom requires several cases, depending on the values of $\varphi_i'(z), \varphi_i'(w)$.  For example, we may have $\varphi_i'(z) = \varphi_i'(w) = 1$ and $\varphi_i'(z)\neq \varphi(z)$, $\varphi_i'(w)\neq \varphi(z)$, giving a diagram of the form:
\begin{align*}
\text{Case (i):} \qquad & \raisebox{-.5\height}{\includegraphics[scale=1]{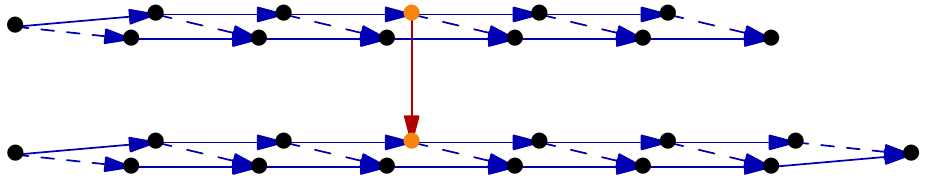}}\\ \\
\text{Case (ii):} \qquad & \raisebox{-.5\height}{\includegraphics[scale=1]{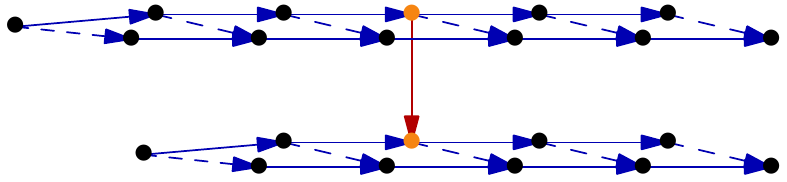}}
\end{align*}
Notice the strong resemblance to Stembridge's Axiom \ref{ax:Stembridge-lengths}.

\subsection{Merge Axioms}

We now state the axioms governing $i/i'$ and $(i+1)/(i+1)'$ edges pointing `downwards' from a common vertex.  In general we will use the convention of drawing $f_i$ and $f_i'$ arrows as pointing downwards, and $e_i$ and $e_i'$ pointing upwards. 

These axioms should be understood as follows. Let $w$ be a vertex with the indicated pair of downwards ($f$) edges, $w \xrightarrow{i\text{ or } i'} x$ and $w \xrightarrow{i+1 \text{ or } (i+1)'} y$. Let $z$ be a common minimum of $x$ and $y$ obtained from a breadth-first search among only $i,i',i+1,(i+1)'$ edges.  That is, $z$ is an element of greatest possible weight among all common lower bounds in this edge-induced subgraph. Then the interval $[w,z] \subset G$ contains the depicted vertices and edges if and only if the specified conditions on $w$ hold.  In this sense, these axioms are analogous to Stembridge axioms \ref{ax:Stembridge-merge} and \ref{ax:Stembridge-dual-merge}. 

Note that $z$ need not be a unique greatest lower bound (e.g. both axioms \ref{ax:primed-square} and \ref{ax:half-solid-square} below may hold). In particular, $G$ need not be a lattice.

In what follows, we define
\[\Delta=(\Delta\veps_i, \Delta\veps_{i+1}) = (\veps_i(w)-\veps_i(y), \veps_{i+1}(w) - \veps_{i+1}(x)).\]
By Axiom \ref{ax:length}, this pair is always either $(0,0)$, $(0,1)$, $(1,0)$ or $(1,1)$.

The first two axioms are for connecting two primed arrows from $w$:

\begin{itemize}
  \item[\mylabel{ax:primed-square}{(A1)}] \textbf{Primed Square:} If $x=f'_i(w)$ and $y=f'_i(w)$ are both defined, then $$f'_i(f'_{i+1}(w))=f_{i+1}'(f_i'(w))\neq \varnothing.$$
  \item[\mylabel{ax:half-solid-square}{(A2)}] \textbf{Half-solid Square:} If $x=f'_i(w)$ and $y=f'_{i+1}(w)$ are both defined, then $$f_{i}(f'_{i+1}(w))=f_{i+1}(f_{i}'(w))\neq \varnothing\hspace{0.3cm}\text{ and }\hspace{0.3cm}f_{i}'(f_{i+1}'(w))\neq f_i(f_{i+1}'(w))$$ if and only if $\Delta=(0,0)$ and $\varphi_{i+1}(w)=1$ and $\widehat{\varphi}_{i+1}(w)=0$.
\end{itemize}

The next two axioms tell us how to merge a primed and unprimed arrow:

\begin{itemize}
  \item[\mylabel{ax:f1primef2-square}{(A3)}] \textbf{Square for $\{f_i',f_{i+1}\}$:}  If $x=f'_i(w)$ and $y=f_{i+1}(w)$ are both defined and either $f_{i+1}'(w)\neq y$ or $f_{i}(w)\neq x$, then $$f'_i(f_{i+1}(w))=f_{i+1}(f_{i}'(w))\neq \varnothing.$$
  \item[\mylabel{ax:f1f2prime-square}{(A4)}] \textbf{Square for $\{f_i,f_{i+1}'\}$:}  If $x=f_i(w)$ and $y=f_{i+1}'(w)$ are both defined, then $$f_i(f_{i+1}'(w))=f_{i+1}'(f_i(w))\neq \varnothing$$ if and only if $\vepshat_i(w)>0$.
\end{itemize}

The final axioms enable us to merge two solid arrows:

\begin{itemize}
  \item[\mylabel{ax:half-primed-square}{(A5)}] \textbf{Half-primed Square:}  Suppose $x=f_i(w)$ and $y=f_{i+1}(w)$ are both defined and $f_i'(w)=\varnothing$. Then $$f_i'(f_{i+1}(w))=f_{i+1}'(f_i(w))\neq \varnothing$$ if and only if $\Delta=(1,1)$.
  \item[\mylabel{ax:square}{(A6)}] \textbf{Square:} Suppose $x=f_i(w)$ and $y=f_{i+1}(w)$ are both defined and $f_i'(w)=\varnothing$.  We have that $$f_i(f_{i+1}(w))=f_{i+1}(f_i(w))\neq \varnothing$$ if and only if either $\Delta=(1,0)$ or $\Delta=(0,1)$. 
  \item[\mylabel{ax:half-primed-stembridge}{(A7)}] \textbf{Half-primed Octagon:} Suppose $x=f_i(w)$ and $y=f_{i+1}(w)$ are both defined and $f_i'(w)=\varnothing$.  We have that $$f_{i+1}(f_i'(f_i(f_{i+1}(w))))=f_i'(f_{i+1}^2(f_i(w)))\neq \varnothing\hspace{0.3cm}\text{ and }\hspace{0.3cm} f_i(f_{i+1}(w))\neq f_{i+1}(f_i(w))$$ if and only if $\Delta=(0,0)$ and $\vepshat_i(w)-\vepshat_i(f_{i+1}(w))=-1$.  
  \item[\mylabel{ax:stembridge}{(A8)}] \textbf{Octagon:} Suppose $x=f_i(w)$ and $y=f_{i+1}(w)$ are both defined and $f_i'(w)=\varnothing$.  We have that $$f_{i+1}(f_i^2(f_{i+1}(w)))=f_i(f_{i+1}^2(f_i(w)))\neq \varnothing\hspace{0.3cm}\text{ and }\hspace{0.3cm}f_i(f_{i+1}(w))\neq f_{i+1}(f_i(w))$$ if and only if $\Delta=(0,0)$ and $\widehat{\varphi}_i(f_{i+1}(w))\ge 2$. 
\end{itemize}

\begin{remark}
The conditions for Axioms \ref{ax:half-primed-stembridge} and \ref{ax:stembridge} are not mutually exclusive, and indeed both can occur together from the same vertex $w$. They are, however, exhaustive, as shown in the next proposition.
\end{remark}

\begin{proposition} \label{prop:some-stembridge-applies}
Let $w$ be a vertex with $f_i,f_{i+1}$ defined and $f'_i$ not defined. If $\Delta = (0,0)$, then at least one of Axioms \ref{ax:half-primed-stembridge} and \ref{ax:stembridge} applies at $w$.
\end{proposition}
\begin{proof}
Since $f_i$ is defined but not $f_i'$, the $i$-string through $w$ is separated; in particular $\varepsilon'_i(w) = 1$. Now let $y = f_{i+1}(w)$. By the $\Delta$ condition, $\varepsilon_i(w) = \varepsilon_i(y)$, and by Axiom \ref{ax:length},
\[\varphi_i(y) = \varphi_i(w) + 1 \geq 2.\]
Now suppose $\widehat{\varphi_i}(y) < 2$. Since $\varphi_i(y) \ne \widehat{\varphi_i}(y)$, the $i$-string through $y$ is separated, and moreover $\widehat{\varphi_i}(y) = \varphi_i'(y) = 1$ and $\varepsilon_i'(y) = 0$. We now see
\[\widehat{\varepsilon_i}(w) - \widehat{\varepsilon_i}(y) = \varepsilon_i(w) - \varepsilon_i'(w) - \varepsilon_i(y) + \varepsilon_i'(y) = -1,\]
as required for Axiom \ref{ax:half-primed-stembridge} .
\end{proof}
Figure \ref{fig:primed-axioms} shows the Axioms \ref{ax:primed-square}--\ref{ax:f1f2prime-square}.  Figure \ref{fig:unprimed-axioms} shows those for \ref{ax:half-primed-square}--\ref{ax:stembridge}.

\begin{figure}[t]
\caption{\label{fig:primed-axioms}Axioms for $\{f_i',f_{i+1}'\}$, $\{f_i',f_{i+1}\}$, and $\{f_i,f_{i+1}'\}$.  Arrows for $f_i$ and $f_i'$ shown in red, those for $f_{i+1}$ and $f_{i+1}'$ shown in blue. \vspace{0.1cm}}
\centering
\begin{tabular}[c]{|c|c|} \hline
Conditions & Axiom \\ \hline
(none) & \raisebox{-.5\height}{\includegraphics[scale=1]{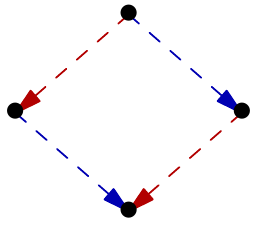}} \\ \hline
\begin{tabular}[c]{@{}c@{}}
$\Delta = (0,0)$ and \\
$\vphi_{i+1}(w) = 1$ and \\
$\vphihat_{i+1}(w)=0$
\end{tabular}&
\raisebox{-.5\height}{\includegraphics[scale=1]{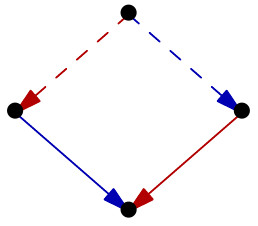}} \\ 
& (and no primed square) \\ \hline
\end{tabular}
\qquad 
\begin{tabular}[c]{|c|c|} \hline
Conditions & Axiom \\ \hline
If $f_{i+1}'(w)\neq f_{i+1}(w)$ & \raisebox{-.5\height}{\includegraphics[scale=1]{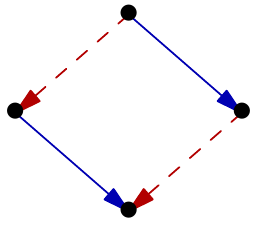}} \\ \hline
\ \ \ \ $\vepshat_i(w) > 0$ \ \ \ \ & \raisebox{-.5\height}{\includegraphics[scale=1]{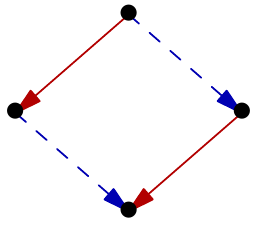}} \\ \hline
\end{tabular}
\end{figure}

\begin{remark} \label{rmk:excluded-lengths}
Note that the axioms do not say anything for $\{f_i, f_{i+1}'\}$ when $\vepshat_i(w) = 0$. Indeed, the length of the path to a common minimum of $f_i(w), f_{i+1}'(w)$ may be arbitrarily long in this case.  
It is somewhat surprising that the uniqueness statement in Theorem \ref{thm:axioms} holds in the absence of such an axiom. In fact, it turns out (Lemma \ref{lem:excluded-lengths}) that one of the other axioms always applies at such a vertex.
\end{remark}

\begin{figure}[t]
\caption{\label{fig:unprimed-axioms}Axioms for $\{f_i,f_{i+1}\}$. Arrows for $f_i$ and $f_i'$ shown in red, those for $f_{i+1}$ and $f_{i+1}'$ shown in blue. \vspace{0.1cm}}
\centering

		\begin{tabular}[c]{|c|c||@{}c@{}|c|} \hline
			Conditions & Axiom & Conditions & Axiom \\ \hline
			$\Delta = (0,1)$, $(1,0)$ & \raisebox{-.5\height}{\includegraphics[scale=1]{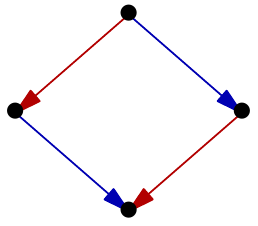}}  &
			$\Delta = (1,1)$ & \raisebox{-.5\height}{\includegraphics[scale=1]{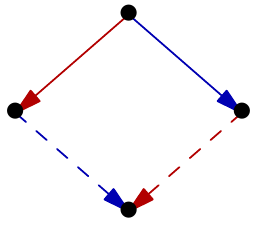}}
			\\ \hline
			\begin{tabular}[c]{c}
				$\Delta = (0,0)$, \\
				$\widehat{\varphi}_i(f_{i+1}(w)) \geq 2$
			\end{tabular}&
			\raisebox{-.5\height}{\includegraphics[scale=1]{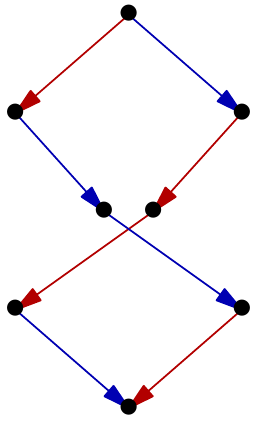}} &
			\begin{tabular}[c]{c}
				$\Delta = (0,0)$, \\
				$\vepshat_1(w) -\vepshat_1(y) = -1$
			\end{tabular}& \raisebox{-.5\height}{\includegraphics[scale=1]{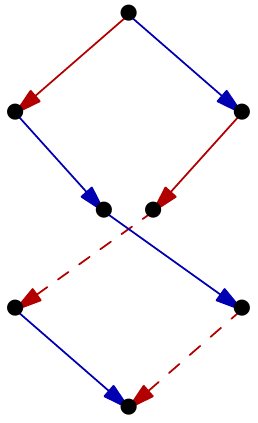}} \\ \hline
		\end{tabular}
\end{figure}

\subsection{The dualized axioms}

We also require the $\eta$-duals of these axioms, obtained by interchanging the roles of $f_i$ with $e_j$, $f_i'$ with $e_j'$, $\varphi_i$ with $\veps_j$, $\hat{\varphi}_i$ with $\hat{\veps}_j$, and $\varphi_i'$ with $\veps_j'$ for $(i,j)\in \{(1,2),(2,1)\}$.  

In particular, the dual axioms start with a node $w$ with two arrows coming into it, say from $x=e_{i+1}(w)$ or $e_{i+1}'(w)$ and $y=e_i(w)$ or $e_i'(w)$.  We define $$\Delta'=(\varphi_{i+1}(w)-\varphi_{i+1}(y),\varphi_i(w)-\varphi_i(x)).$$   Then dual of axioms are as follows:

\begin{itemize}
  \item[\mylabel{ax:dual-primed-square}{(A1$^\ast$)}] \textbf{Dual Primed Square:} If $x=e'_{i+1}(w)$ and $y=e'_i(w)$ are both defined, then $$e'_{i+1}(e'_i(w))=e_i'(e_{i+1}'(w))\neq \varnothing.$$
  \item[\mylabel{ax:dual-half-solid-square}(A2$^\ast$)] \textbf{Dual Half-solid Square:} If $x=e'_{i+1}(w)$ and $y=e'_i(w)$ are both defined, then $$e_{i+1}(e'_i(w))=e_i(e_{i+1}'(w))\neq \varnothing\hspace{0.3cm}\text{ and }\hspace{0.3cm}e_{i+1}'(e_i'(w))\neq e_{i+1}(e_i'(w))$$ if and only if $\Delta'=(0,0)$ and $\veps_{i}(w)=1$ and $\vepshat_i(w)=0$.
  \item[\mylabel{ax:dual-f1primef2-square}{(A3$^\ast$)}] \textbf{Dual Square for $\{e_{i+1}',e_i\}$:}  If $x=e'_{i+1}(w)$ and $y=e_i(w)$ are both defined and either $e_i'(w)\neq y$ or $e_{i+1}(w)\neq x$, then $$e'_{i+1}(e_i(w))=e_i(e_{i+1}'(w))\neq \varnothing.$$
  \item[\mylabel{ax:dual-f1f2prime-square}{(A4$^\ast$)}] \textbf{Dual Square for $\{e_{i+1},e_i'\}$:}  If $x=e_{i+1}(w)$ and $y=e_i'(w)$ are both defined, then $$e_{i+1}(e_i'(w))=e_i'(e_{i+1}(w))\neq \varnothing$$ if and only if $\widehat{\varphi}_{i+1}(w)>0$.
  \item[\mylabel{ax:dual-half-primed-square}{(A5$^\ast$)}] \textbf{Dual Half-primed Square:}  Suppose $x=e_{i+1}(w)$ and $y=e_i(w)$ are both defined and $e_{i+1}'(w)=\varnothing$. Then $$e_{i+1}'(e_i(w))=e_i'(e_{i+1}(w))\neq \varnothing$$ if and only if $\Delta'=(1,1)$.
  \item[\mylabel{ax:dual-square}{(A6$^\ast$)}] \textbf{Dual Square:} Suppose $x=e_{i+1}(w)$ and $y=e_i(w)$ are both defined and $e_{i+1}'(w)=\varnothing$.  We have that $$e_{i+1}(e_i(w))=e_i(e_{i+1}(w))\neq \varnothing$$ if and only if either $\Delta'=(1,0)$ or $\Delta'=(0,1)$. 
  \item[\mylabel{ax:dual-half-primed-stembridge}{(A7$^\ast$)}] \textbf{Dual Half-primed Octagon:} Suppose $x=e_{i+1}(w)$ and $y=e_i(w)$ are both defined and $e_{i+1}'(w)=\varnothing$.  We have that $$e_i(e_{i+1}'(e_{i+1}(e_i(w))))=e_{i+1}'(e_i^2(e_{i+1}(w)))\neq \varnothing\hspace{0.3cm}\text{ and }\hspace{0.3cm} e_{i+1}(e_i(w))\neq e_i(e_{i+1}(w))$$ if and only if $\Delta'=(0,0)$ and $\widehat{\varphi}_{i+1}(w)-\widehat{\varphi}_{i+1}(e_i(w))=-1$.  
  \item[\mylabel{ax:dual-stembridge}{(A8$^\ast$)}] \textbf{Dual Octagon:} Suppose $x=e_{i+1}(w)$ and $y=e_{i}(w)$ are both defined and $e_{i+1}'(w)=\varnothing$.  We have that $$e_i(e_{i+1}^2(e_i(w)))=e_{i+1}(e_i^2(e_{i+1}(w)))\neq \varnothing\hspace{0.3cm}\text{ and }\hspace{0.3cm}e_{i+1}(e_i(w))\neq e_i(e_{i+1}(w))$$ if and only if $\Delta'=(0,0)$ and $\vepshat_{i+1}(e_i(w))\ge 2$. 
\end{itemize}

\section{Unique characterization by the axioms} \label{sec:proof-of-uniqueness}

In this section, we prove part of Theorem \ref{thm:axioms} -- that the local axioms stated in Section \ref{sec:axioms} determine the global structure of the graph.

\begin{definition}
We say a graph $G$ {\bf satisfies the axioms} if $G$ is a finite directed graph with vertices weighted by a function $\wt: G \to \mathbb{Z}_{\geq 0}^n$ and edges labeled $1', 1, \ldots, n{-}1', n{-}1$, satisfying the axioms (B1)-(B3), \ref{ax:Kashiwara-B}, \ref{ax:primed-square}-\ref{ax:stembridge}, and \ref{ax:dual-primed-square}-\ref{ax:dual-stembridge} in Section \ref{sec:axioms}.
\end{definition}

We will show that if $G$ satisfies the axioms, and is connected, then it has a unique highest weight element $g$, whose weight is a strict partition (Proposition \ref{prop:unique-maximum}). We then show that $G$ as a whole is uniquely determined by the weight of the maximal element (Proposition \ref{prop:uniquely-determined}). It remains to show that such a graph actually exists, for each strict partition. We address this in Section \ref{sec:tableaux}, where we show that the shifted tableau crystals $G = \ShST(\lambda,n)$ satisfy the axioms.

\subsection{Preliminary lemmas}

We first require a technical lemma that shows the axioms of Section \ref{sec:axioms} prevent a certain combination of lengths from occurring simultaneously at a node. This case is precisely the situation in which none of the merge axioms would apply (see Remark \ref{rmk:excluded-lengths}). It follows from the lemma below that our axioms do indeed `cover all cases', which we need for the uniqueness proof.

\begin{lemma}[Excluded lengths]\label{lem:excluded-lengths}
Suppose $G$ satisfies the axioms. If a vertex $w \in G$ has $\widehat{\varepsilon}_i(w) = \varphi'_i(w) = \widehat{\varphi}_{i+1}(w) = \varepsilon'_{i+1}(w) = 0$, then there are no $i$, $i+1$, $i'$, or $(i+1)'$ edges connected to $w$.  
\end{lemma}

\begin{proof}
It suffices to assume $i=1$ and $n=3$, by Axiom \ref{ax:basic2}.  So, suppose $w$ satisfies $\widehat{\varepsilon}_1(w) = \varphi'_1(w) = \widehat{\varphi}_{2}(w) = \varepsilon'_{2}(w) = 0$. Let $\wt(w) = (a,b,c)$.  Let $G$ be the connected component of $w$.  In this setting we wish to show that $G=\{w\}$.  

{\bf Step 1 (A primed square at $w$)}.
We first claim that if $G\neq \{w\}$, then $x = e_1'(w) \ne \varnothing$. To see this, note that the only way for a vertex to have $e'_1(w) = f'_1(w) = \varnothing$ is if the $1/1'$ string through $w$ consists of a single vertex, by Axiom \ref{ax:basic1}. Moreover, in that case $\wt(w) = (0,0,c)$, and in particular, the $2/2'$ string through $w$ is also collapsed. But then the assumed conditions $\widehat{\varphi}_{2}(w) = \varepsilon'_{2}(w) = 0$ show that the $2/2'$ string is also a single vertex (and $c=0$). Then, since $G$ is connected, $G = \{w\}$.

By a symmetric argument, we also deduce $y = f_2'(w) \ne \varnothing$ (unless $G = \{w\}$).  Moreover, the assumed conditions $\vepshat_1(w)=0$ and $\vphihat_2(w)=0$ now show that both strings through $w$ are separated strings by Axiom \ref{ax:basic1}, and that $x$ is the top of the 1/1' string through $w$, and $y$ the bottom of the $2/2'$ string through $w$. Since all weights are in $\mathbb{Z}_{\ge 0}^3$, we conclude by Axiom \ref{ax:Kashiwara-B} that $b \geq 2$.

Finally, we claim that $z = f_2'(x) \ne \varnothing$. For contradiction, suppose not. It follows that $f_2(x) = \varnothing$ as well, since otherwise (noting that $f_2 \ne f_2'$ at $x$) Axiom \ref{ax:f1primef2-square} would imply that $f_2(w)$ is defined, contradicting our assumptions. So $x$ is the bottom of a $2/2'$ string, which by Axiom \ref{ax:basic1} is not collapsed, since $\wt(x) = (a+1,b-1,c)$ by Axiom \ref{ax:Kashiwara-B} and $b-1 > 0$. Therefore $\tilde{w} = e_2'(x)$ must be defined. We claim that $\tilde{w}$ satisfies the assumed conditions at $w$, namely, $\widehat{\varepsilon}_1(\tilde{w}) = \varphi'_1(\tilde{w}) = \widehat{\varphi}_{2}(\tilde{w}) = \varepsilon'_{2}(\tilde{w}) = 0$, giving a contradiction by infinite regress. We already see $\widehat{\varphi}_2(\tilde{w}) = \varepsilon'_2(\tilde{w}) = \varnothing$. For the $1/1'$ string, Axiom \ref{ax:length} applied to the $\tilde{w} \xrightarrow{2'} x$ edge shows that either $\tilde{w}$ is the top of the $1/1'$ string, or the top is one step higher. Either way, the top has $\wt_2 = b$ or $b-1 > 0$, so the $1/1'$ string is separated. We also can't have $f_1'(\tilde{w})$ defined, since then Axiom \ref{ax:primed-square} would make $e_2'(w)$ defined (equalling $f_1'(\tilde{w})$), contradicting the assumption $\veps_2'(w)=0$. Therefore $e_1'(\tilde{w})$ is defined instead and, since this is the top of the string, $e_1(\tilde{w})$ can't be defined. This gives the desired properties for $\tilde{w}$.

We have shown that there is a primed square,
\begin{center}\includegraphics[scale=1]{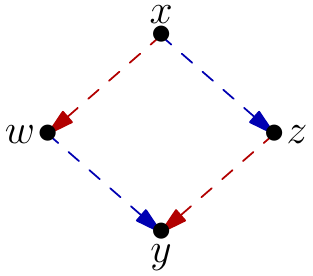}\end{center}
where the strings through $w$ are separated, with top and bottom vertices $x$ and $y$ respectively.

{\bf Step 2 (Setting up vertex $s := e_2(x)$ for contradiction)}.  Since $y$ is at the bottom of its $2/2'$ string, Axiom \ref{ax:length} applied to the $z \xrightarrow{1'} y$ edge shows that $\varepsilon_2(z) \geq 2$.  Thus $s = e_2(x)$ is defined.

We now analyze $s$ in detail.  By Axiom \ref{ax:length} applied to the $s \xrightarrow{2} x$ edge, either $\veps_1(s)=0$ or $\veps_1(s)=1$.  Either way, $\wt_2 = b$ or $b-1 > 0$ at the top of the $1/1'$ string through $s$, so the string is separated by Axiom \ref{ax:basic1}. In particular, exactly one of $f_1'$ or $e_1'$ is defined at $s$.

Suppose $f_1'(s)$ is defined. If $f_2(s) = f_2'(s)$, we get a primed square with top element $s$ and bottom element $w$ by Axiom \ref{ax:primed-square}, implying $e_2'(w)$ defined, a contradiction. If instead $f_2(s)\neq f_2'(s)$, Axiom \ref{ax:f1primef2-square} gives an $\{f_1',f_2\}$ square from $s$ to $w$ (involving $x$), contradicting the \emph{dual} axiom \ref{ax:dual-f1f2prime-square} since $\widehat{\varphi}_2(w) = 0$. Either way we get a contradiction, so $f_1'(s)$ is not defined.

Therefore $t = e_1'(s)$ is defined and is the top of the separated $1/1'$ string through $s$. (See Figure \ref{fig:s-setup}.) Axiom \ref{ax:length}, applied to the $s \xrightarrow{2} x$ edge, now shows that $\varphi_1(s) = \varphi_1(x)$. Comparing the location of $s$ and $x$ in their respective separated$1/1'$ strings, we see that
\[\widehat{\varphi}_1(s) = \widehat{\varphi}_1(x) + 1 = \widehat{\varphi}_1(w) + 1 = N+1\]
where $N=\vphi_1(w)$.

\begin{figure}
\centering
\includegraphics{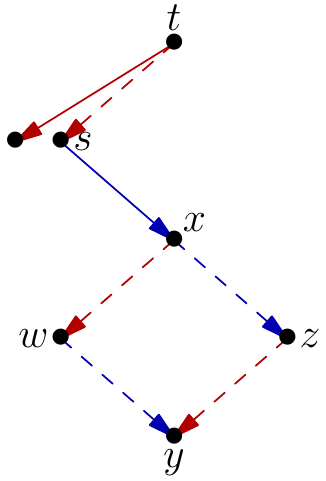}
\caption{\label{fig:s-setup} The vertices $s$ and $t$ defined in Step 2 of the proof by contradiction of Lemma \ref{lem:excluded-lengths}.}
\end{figure}

{\bf Step 3 (Contradiction from $\{f_1,f_2\}$ squares at $s$).} We claim that at each of the vertices $s, f_1(s), f_1^2(s), \ldots, f^N(s)$, there is an $f_2$ edge pointing down from it and an $\{f_1, f_2\}$ square attaching this edge to the $1/1'$ string through $x$. This gives a contradiction since then $f_2f_1^{N+1}(s) = f_1^{N+1}(x) = \varnothing$. 

We prove this claim inductively. For the base case, observe that since $t$ and $x$ are at the top of their respective $1/1'$ strings, we have 
\[\varepsilon_1(s) - \varepsilon_1(x) = 1.\]  Since $f_1(s)$ is defined, as is $f_2(s)=x$ (but not $f_1'$), we either get an $\{f_1,f_2\}$ square or half-primed square with top vertex $s$ by Axioms \ref{ax:half-primed-square} and \ref{ax:square}. But in the latter case $e_2'(w)$ would be defined, contradicting the assumption $\veps_2'(w)=0$.

Next, suppose the claim is established for $f_1^i(s)$ for some $0 \leq i \leq N-1$. Write $\tilde{s} = f_1^{i+1}(s)$. We have $f_1(\tilde{s})$ defined since $i+1 < N+1$ (and $f_1'(\tilde{s})$ not defined), and inductively we have 
\[\tilde{x} := f_2(\tilde{s}) = f_2f_1(f_1^i(s))\]
is defined and equals 
\[f_1f_2(f_1^i(s)) = \cdots = f_1^{i+1}(f_2(s)) = f_1^{i+1}(x).\]
Now we see directly that 
\[\varepsilon_1(\tilde{s}) - \varepsilon_1(\tilde{x}) = (i+2) - (i+1) = 1,\]
so again we get either an $\{f_1,f_2\}$ square or half-primed square at $\tilde{s}$. But if we get a half-primed square, we see that 
\[f_2'f_1(\tilde{s}) = f_1'f_1^{i+1}(x) = f_1^{i+1}(w),\]
so $e_2'$ is defined at $f_1^{i+1}(w)$. But then $e_2'$ is defined at $w$ by repeated application of the dual axiom \ref{ax:dual-f1primef2-square} (note that the $f_1$ edge is separated), contradicting our assumption that $\veps_2'(w)=0$. 
This completes the proof.
\end{proof}

We will also use the following simpler lemmas. For simplicity we state them only for $i=1,2$:

\begin{lemma} \label{lem:collapsed-adjacent-strings}
Suppose $G$ satisfies Axioms \ref{ax:basic1}-\ref{ax:length} and \ref{ax:Kashiwara-B}.
\begin{itemize}
\item[(i)] If $w \xrightarrow{1 \text{ or } 1'} z$ is an edge and $\varphi_{2}(z) = \varphi_{2}(w)+1$, then the $2/2'$-string through $w$ is collapsed if and only if the $2/2'$-string through $z$ is collapsed.
\item[(ii)] If $w \xrightarrow{1 \text{ or } 1'} z$ is an edge and $\varphi_{2}(z) = \varphi_{2}(w)$, the $2$-string through $z$ is separated.
\item[(iii)] If $w \xrightarrow{2 \text{ or } 2'} z$ is an edge, the $1$-string through $w$ is collapsed if and only if the $1$-string through $z$ is collapsed and $\varphi_1(z) = \varphi_1(w)$.
\end{itemize}
\end{lemma}
\begin{proof}
For (i) and (ii), this follows because $\wt(z)_2 = \wt(w)_2 + 1$ and because a $2$-string is collapsed if and only if $\wt_2 = \varphi_2$ along the string (since the bottom vertex has $\wt_2 = 0$). 

For statement (iii)$(\Rightarrow)$: we have $\wt(z)_1 = \wt(w)_1 = \varphi_1(w)$, and we cannot have $\varphi_1(z) = \varphi_1(w)+1$ since the weights must remain nonnegative at the bottom of the $1$-string. Thus $\varphi_1(z) = \varphi_1(w) = \wt(z)_1$ and the $1$-string through $z$ is collapsed. For (iii)($\Leftarrow$), we have $\wt(z)_1 = \wt(w)_1 = \varphi_1(z) = \varphi_1(w)$.
\end{proof}

\begin{lemma}\label{lem:copying-along-f1'}
Suppose $G$ satisfies the axioms. If $w \xrightarrow{1'} z$ is an edge and $\varphi_2(w) = \varphi_2(z)$, then $\widehat{\varphi}_2(z) = 0$.
\end{lemma}

\begin{proof}
By Lemma \ref{lem:collapsed-adjacent-strings}(ii), the $2$-string through $z$ is separated.

If the $2$-string through $w$ is collapsed, we see by repeated application of Axiom \ref{ax:primed-square} that $\varphi_2'(w) \leq \varphi_2'(z)\ (\leq 1)$. But then we see 
\begin{align*}
\varphi_2'(z) \leq \varphi_2'(z) + \widehat{\varphi}_2(z) &= \varphi_2(z) \\
&= \varphi_2(w) = \varphi_2'(w) \leq \varphi_2'(z),
\end{align*}
so $\widehat{\varphi}_2(z) = 0$ as desired.

Now suppose the $2$-string through $w$ is separated. By repeated application of Axioms \ref{ax:primed-square} (if $\varphi_2'(w)= 1$) and \ref{ax:f1primef2-square}, we see that $\varphi_2'(w) \leq \varphi_2'(z)$ and $\widehat{\varphi}_2(w) \leq \widehat{\varphi}_2(z)$. By hypothesis the sums $\varphi_2' + \widehat{\varphi}_2$ are equal at $w$ and $z$, so in fact both inequalities are equalities. Suppose for contradiction $\widehat{\varphi}_2(z) = \widehat{\varphi}_2(w) = 1 + N$, with $N \geq 0$. Let $\tilde{w} :=  f_2^N(w)$. By repeated application of Axiom \ref{ax:f1primef2-square}, we have $f_2f_1'(\tilde{w}) = f_1'f_2(\tilde{w}) = f_2^{1+N}(z)$, which contradicts the dual axiom \ref{ax:dual-f1f2prime-square} since $\widehat{\varphi}_2(f_2^{N+1}(z)) = 0$.
\end{proof}

\begin{lemma}[Translating $\Delta$] \label{lem:copying-lengths-along-f1'}
Suppose $G$ satisfies the axioms and $z$ is a vertex in $G$ with $f_1'(z)=t,f_1(z)=x,f_2(z)$ defined and $t \ne x$. By Axioms \ref{ax:basic1} and \ref{ax:f1primef2-square}, $f_1(t)$ and $f_2(t)$ are defined. Then for the $\{f_1,f_2\}$ edges at $z$ and at $t$, $\Delta_z = \Delta_t$.
\end{lemma}
\begin{proof}
By Axiom \ref{ax:f1primef2-square}, there is an $\{f_1',f_2\}$ square at $z$, from which $\Delta\varepsilon_1$ is the same along the $f_2$ edges at $t$ and at $z$. 

Since $f_2$ is defined at $t$, we have $\varphi_2(t)\neq \varphi_2(z)$ by Lemma \ref{lem:copying-along-f1'}.  Hence by Axiom \ref{ax:length} we have $\varepsilon_2(t) = \varepsilon_2(z)$.

If $\Delta\varepsilon_2 = 1$ at $z$, we get $\varepsilon_2(x) = \varepsilon_2(z) - 1$. Since $f_1(t) = f_1'(x)$ and $\varepsilon_2$ either decrements or stays the same along any $f_1$ and $f_1'$ edge, we must have $\varepsilon_2(f_1(t)) = \varepsilon_2(t) - 1$, so $\Delta\varepsilon_2 = 1$ at $t$.

We are left with the possibilities $\Delta_z = (1,0)$ but $\Delta_t = (1,1)$; and $\Delta_z = (0,0)$ but $\Delta_t = (0,1)$. In the first scenario, Axiom \ref{ax:half-primed-square} shows that $f_1'(f_2(t)) = f_1'(f_1'(f_2(z))$ is defined, so the $1$-string through $f_2(t)$ is collapsed; also, $\varphi_1(t) = \varphi_1(f_2(t))$ since $\Delta\varepsilon_1 = 1$ at $t$. But then the $1$-string through $t$ is collapsed by Lemma \ref{lem:collapsed-adjacent-strings}(iii), contradicting our hypotheses. In the second scenario, we obtain a square $f_1f_2(t) = f_2f_1(t)$ by Axiom \ref{ax:square}. In particular $f_2$ is defined at $r = f_1(t) = f_1'(x)$, so a final application of Lemma \ref{lem:copying-along-f1'} and Axiom \ref{ax:length} to the edge $x \xrightarrow{f_1'} r$ shows $\varepsilon_2(r) = \varepsilon_2(x)$. We have $\varepsilon_2(x) = \varepsilon_2(z)$ since $\Delta\varepsilon_2 = 0$ at $z$, and we know $\varepsilon_2(z) = \varepsilon_2(t)$, so $\Delta\varepsilon_2 = 0$ at $t$.
\end{proof}

\subsection{Existence of a unique highest weight element}

Let $G$ be a finite, nonnegatively weighted, edge-labeled directed graph as above.  Define a \textbf{highest weight element} of $G$ to be an element $g$ for which there are no incoming arrows pointing to $g$.  In other words, if $G$ satisfies the axioms, $g$ is highest weight if and only if $\veps_i(g)=0$ for all $i$.  Note that every element of $G$ is comparable to some highest weight element, since repeated applications of the $e_i$ operators will eventually terminate.

\begin{proposition} \label{prop:unique-maximum}
If $G$ is connected and satisfies the axioms 
, then $G$ has a unique highest weight element $g \in G$, whose weight $\wt(g)=(a_1,\ldots,a_n)$ satisfies $a_i\ge a_{i+1}$ for all $i$, and if $a_{i}>0$ then the inequality is strict, that is, $a_i>a_{i+1}$.
\end{proposition}

\begin{proof}
For contradiction, let $w$ be maximal among vertices of $G$ that are comparable to two or more distinct highest weight elements. Then there are elements $x$ and $y$ with arrows $x \rightarrow w \leftarrow y$, for which each of $x$, $y$ comparable to a single highest weight element, say $m_x$ and $m_y$.  
We consider the possible cases, based on the types of the arrows $x \rightarrow w \leftarrow y$.

If the two arrows are labeled $i$ (or $i'$) and $j$ (or $j'$) where $|i-j|>1$, then by Axiom \ref{ax:basic2} the arrows commute up into a square, say with top vertex $z$.  Then $z$ is a common upper bound of $x$ and $y$, and so the unique highest weight elements $m_x$ above $x$ and $m_y$ above $y$ also are comparable to $z$.  This contradicts maximality of $w$.  Similarly, if $i=j$ then Axiom \ref{ax:basic1} shows that $x$ and $y$ have a common upper bound as well, and we get a contradiction.

It remains to consider the case when $i=j\pm 1$, and for simplicity we assume $i=1$ and $j=2$.  If the two arrows are $\{f_1', f_2'\}$ or $\{f_1, f_2'\}$, then the dual axioms \ref{ax:dual-primed-square} and \ref{ax:dual-f1primef2-square} imply the existence of a vertex $z$ with $x \leftarrow z \rightarrow y$ of the appropriate types of arrow.  In each case, we again have a contradiction.

If the two arrows are $\{f_1, f_2\}$, with $x=e_2(w)$ and $y=e_1(w)$, first suppose $t:=e_2'(w)$ is defined.  Then \ref{ax:dual-f1primef2-square} or \ref{ax:dual-primed-square} (depending on whether $e_1'(w)=x$) imply that $t$ and $y$ have a common upper bound $z$ forming a square with $t$, $y$, and $w$.  If $t=x$ we are done.  Otherwise, by Axiom \ref{ax:basic1} we have that $s:=e_2(t)$ is defined and $f_1'(s)=x$.  Then $s$ and $z$ must have a common upper bound since $t$ is higher than $w$ and therefore is comparable to a unique highest weight element.  Since $z$ covers $y$, it follows that $x$ and $y$ have a common upper bound.  

Now we can assume $e_2'(w)=\varnothing$. The cases $\Delta'=(1,1)$, $\Delta'=(0,1)$, and $\Delta'=(1,0)$ are now covered by axioms \ref{ax:dual-half-primed-square} and \ref{ax:dual-square}.  Finally, if $\Delta' = (0,0)$, then at least one of Axiom \ref{ax:dual-stembridge} or \ref{ax:dual-half-primed-stembridge} applies by (the dual of) Proposition \ref{prop:some-stembridge-applies}. 

Finally, suppose the two arrows are $\{f_1', f_2\}$. If $\vphihat_2(w) > 0$, then we get a commuting square as above by Axiom \ref{ax:dual-f1f2prime-square}. Next, suppose $e_2'(w)$ is defined. By the previous cases, $e_2'(w)$ has a common upper bound both with $x$ and with $y$, so again we are done. The same holds if $e_1(w)$ is defined.  Moreover, if $f_1'(w)$ is defined, then since $e_1'(w)=x$ is also defined, the $1/1'$ string through $w$ is collapsed, so $e_1(w)$ is defined and we are done as before.

Thus we can assume $\vphihat_2(w) = \veps'_2(w) = \vepshat_1(w) = \vphi'_1(w) = 0.$  By Lemma \ref{lem:excluded-lengths}, it follows that $x$ and $y$ do not exist, a contradiction.

We have shown that there is a unique maximum $g \in G$. To show that the entries of $\mathrm{wt}(g)$ are strictly decreasing, Axiom \ref{ax:Kashiwara-B} gives us
\[\varphi_i(g) - \veps_i(g) = \langle\mathrm{wt}(g), \alpha_i\rangle = \mathrm{wt}_i(g) - \mathrm{wt}_{i+1}(g)\]
for all $i$.  Moreover, we have $\veps_i(g)=0$ for all $i$ since $g$ is highest weight, and since $\varphi_i(g) \geq 0$ for all $i$ the weak inequality $\wt_i(g)\ge \wt_{i+1}(g)$ follows.  Moreover, suppose $\wt_i(g)>0$.   Then if $\wt_{i+1}(g)=0$ we have $\wt_i(g)>\wt_{i+1}(g)$. If instead $\wt_{i+1}(g)>0$, by Axiom \ref{ax:basic1} the $i/i'$ string through $g$ is separated since $g$ is at the top of this string, and it follows that $\varphi_i(g)>0$, and again we conclude the strict inequality $\wt_i(g)>\wt_{i+1}(g)$.
\end{proof}

\subsection{The proof of uniqueness}

We now show that the axioms uniquely determine the graph $G$.

\begin{proposition} \label{prop:uniquely-determined}
If $G$ and $G'$ are connected and satisfy the axioms, 
 and their unique highest weight elements $g, g'$ have the same weight $\lambda$, then $G$ is canonically isomorphic to $G'$.
\end{proposition}

\begin{proof}
Let $G_{\leq r}$ denote the subset of $G$ that can be reached from $g$ using at most $r$ steps, and $G_r$ the subset reachable in exactly $r$ downwards steps. Let $G'_{\leq r}, G'_r$ be defined similarly from $g'$. We claim that there is a unique isomorphism $T_{\leq r} : G_{\leq r} \to G'_{\leq r}$ of weighted, edge-labeled graphs, which moreover preserves the quantities $\vepshat_i, \veps'_i, \vphihat_i, \vphi'_i$ (and therefore $\veps_i, \vphi_i$ as well) for all $i$.  Since $G, G'$ are finite, we will obtain an isomorphism $T:G\to G'$ by taking $r \gg 0$.

We proceed by induction. For $r=0$, we set $T(g) = g'$. The $\veps, \varphi$ values are preserved due to Axiom \ref{ax:Kashiwara-B} (since $\veps_i(g) = \veps_i(g') = 0$ for all $i$ and $\mathrm{wt}(g) = \mathrm{wt}(g')$).  For the statistics $\veps',\varphi', \vepshat, \widehat{\varphi}$, Axiom \ref{ax:basic1} tells us whether the strings are collapsed or separated based on the weight, and so these statistics are also equal at $g$ and $g'$.

Now consider $r \geq 1$.  We will construct $T_{\leq r}$ from $T_{\leq r-1}$, and then check that it preserves lengths.  By abuse of notation, we will often write $T$ for $T_{\leq r}$.

 \textbf{Step 1: Setup and reduction to two labels.} First, since $T_{\leq r-1}$ preserves lengths, we see that $G_{r-1}$ has the same set of downwards-pointing edges as $G'_{r-1}$. To construct $T_{\leq r}$, it suffices to show that whenever two such edges have the same target in $G$, their isomorphic images have the same target in $G'$. So, suppose we have $x \rightarrow w \leftarrow y$ in $G$, with $x, y \in G_{r-1}$ and $w \in G_r$.

  If the two edges are labeled $f_{i}$ (or $f_{i}'$) and $f_j$ (or $f_{j}'$) for either $|i-j|>1$ or $i=j$, then Axioms \ref{ax:basic1} and \ref{ax:basic2} show that we can `complete the square,' obtaining $z \in G_{r-2}$ with edges $x \leftarrow z \rightarrow y$.  Then applying these axioms to $T(z),T(x),T(y)$ allows us to complete the square in $G'$ to find an entry $T(w)$.  
  
  We have now reduced to the case that the two arrows have indices $i$ (or $i'$) and $j$ (or $j'$) with $|i-j|=1$.  For simplicity we assume $i=1$ and $j=2$ in the remainder of the proof.
 
 \textbf{Step 2: The cases $\{f_1',f_2'\}$ and $\{f_1,f_2'\}$.}    If the arrows are $\{f'_1,f'_2\}$, Axiom \ref{ax:dual-primed-square} allows us to complete the square as above, again defining the unique image $T(w)$ in $G'$.

  If the two edges are $\{f_1, f_2'\}$, with $f_1(y)=f_2'(x)=w$, the situation is similar, with a few subtleties.   First, if $f_1'(y)=f_1(y)=w$, we can form a primed square to obtain $z$ and then form a primed square again from $T(z)$ in $G'$ to obtain $T(w)$ (using Axioms \ref{ax:dual-primed-square} and \ref{ax:primed-square}).  Since the values of $\widehat{\vphi}_1,\vphi_1',\vphi_1$ are preserved under $T_{\le r-1}$, we also have that $\vphihat_1(T(y))=\vphi'_1(T(y))\ge 1$ and so there is also a collapsed edge at $T(y)$, with $f_1(T(y))=w$.  Otherwise, if $f_1'(y)\neq f_1(y)$, then Axiom \ref{ax:dual-f1primef2-square} applies to yield $z$ as above. Then by Axiom \ref{ax:f1f2prime-square} we must have $\vepshat_1(z) > 0$, since the appropriate square exists at $z$. Since $T_{\leq r-1}$ preserves lengths, we get $\vepshat_1(T(z)) > 0$, so we may complete the square in $G'$.

 \textbf{Step 3: Reducing to $e_2'(w)=\varnothing$ in the case $\{f_1,f_2\}$.} If the two edges are $\{f_1, f_2\}$, write $w = f_1(y) = f_2(x)$ and let \[\Delta'_w = (\varphi_2(w) - \varphi_2(y),\varphi_1(w) - \varphi_1(x)).\]  

First suppose that $e_2'(w)$ is defined.  If $e_2'(w)=e_2(w)=x$ then using Axiom \ref{ax:dual-f1primef2-square} and \ref{ax:f1f2prime-square}, or \ref{ax:dual-primed-square} and \ref{ax:primed-square} if $e_1'(w)=y$, we can extend our bijection as above.   If instead $e_2'(w)\neq e_2(w)$, let $t=e_2'(w)$. Then by the arguments in Step 1 and Step 2, the $f_2'$ arrow from $T(t)$ agrees with the arrows from $T(x)$ and $T(y)$, so those two agree with each other. 

  We can now assume that $e_2'(w)=\varnothing$, as required for the axioms \ref{ax:dual-half-primed-square}-\ref{ax:dual-stembridge}.  We consider several cases based on the value of $\Delta'$. 

\textbf{Step 4: The case $\{f_1,f_2\}$ for $\Delta'_w\neq (0,0)$.}   If $\Delta'_w = (1,1)$, we obtain a dual half-primed square, say with top element $z$, from Axiom \ref{ax:dual-half-primed-square}. Then $f_1(f_2'(z))=f_2(f_1'(z))$, forming a half-solid square at $z$.  By the induction hypothesis, the conditions for Axiom \ref{ax:half-solid-square} for a half-solid square also hold at $T(z)$ in $G'$, allowing us to complete the square and obtain $T(w)$ in $G'$.

Suppose $\Delta_w' = (0,1)$ or $(1,0)$, so that there exists a vertex $z = e_2(y) = e_1(x)$. We first see $\Delta_z = \Delta_w'$ by applying Axiom \ref{ax:length} to the lengths in the square.  If $f_1'(z)$ is not defined, since $\Delta_z = \Delta_w'=(1,0)$ or $(0,1)$, the same holds at $T(z)$ and we may complete the square in $G'$. If $f_1'(z) = t \ne f_1(z)$, then in $G'$, we have $\Delta_{T(z)} = (0,1)$ or $(1,0)$, hence the same holds at $T(t)$ by Lemma \ref{lem:copying-lengths-along-f1'}. In particular, $f_1(T(t))$ is defined by Axiom \ref{ax:basic1} and $f_2(T(t))$ is defined by Axiom \ref{ax:f1primef2-square} applied to $z$, so $v:=f_1f_2(T(t))= f_2f_1(T(t))$ by Axiom \ref{ax:square}.  Note also that $v=f_1f_1'(T(y))=f_1'f_1(T(y))$, so we now may apply Axiom \ref{ax:f1primef2-square} at $T(x)$ to fill in the missing edge $T(x) \xrightarrow{1} f_2f_1(T(z))$ in $G'$.  (See Figure \ref{fig:Delta10}.)

We are left with the case $f_1'(z) = f_1(z)=x$. Lemma \ref{lem:collapsed-adjacent-strings}(iii) shows the $1$-string through $y$ is also collapsed and $\varphi_1(z) = \varphi_1(y)$. Thus we must be in the case $\Delta_z = \Delta_w = (1,0)$, and so $\varphi_2(w) = \varphi_2(y)+1$. Since the $2$-string through $w$ is separated, so is the $2$-string through $y$ by Lemma \ref{lem:collapsed-adjacent-strings}(i). The corresponding conditions hold for the edges $T(x) \xleftarrow{1=1'} T(z) \xrightarrow{2} T(y)$, so Axiom \ref{ax:f1primef2-square} yields an $\{f_1',f_2\}$ square with collapsed $1$-edges at $T(z)$. In particular $f_1f_2(T(z)) = f_2f_1(T(z))$.



\begin{figure}
\begin{center}
\includegraphics[scale=1]{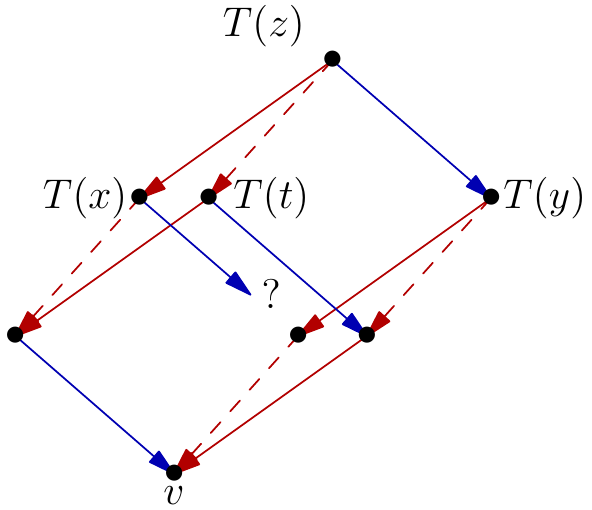}
\end{center}
\caption{\label{fig:Delta10} The construction of $T(w)$ using the auxiliary element $v$ in Step 4.}
\end{figure}

\textbf{Step 5: The case $\{f_1,f_2\}$ for $\Delta'_w=(0,0)$, subcase \ref{ax:dual-stembridge}.} Suppose $\Delta_w' = (0,0)$.  By (the dual of) Proposition \ref{prop:some-stembridge-applies}, at least one of Axiom \ref{ax:dual-half-primed-stembridge} or \ref{ax:dual-stembridge} applies at $w$. We consider \ref{ax:dual-stembridge} first. 

Let $z$ be the top element of the dual octagon shape from Axiom \ref{ax:dual-stembridge} at $w$.  If $f_1'(z)=\varnothing$ then by Axiom \ref{ax:stembridge},
\[\Delta_z := (\veps_1(z)-\veps_1(f_2(z)),\veps_2(z)-\veps_2(f_1(z))) = (0,0) \text{ and } \widehat{\varphi}_1(f_2(z)) \geq 2.\] These conditions hold at $T(z)$ by induction and we're done. 

Otherwise, suppose $f_1'(z)\neq \varnothing$.  Then an octagon pattern holds at $T(z)$ except possibly with the last two arrows from $T(x)$ and $T(y)$ pointing to different vertices (see Figure \ref{fig:special-stembridge-case}).

Let $a=f_1(T(z))$ and $b=f_2(T(z))$.  First assume for contradiction that $f_1'(T(z))=f_1(T(z))=a$.  Then if $f_2'(T(z))=b$ we have a primed square at $T(z),a,b,c$ where $c=f_1'(b)=f_2'(a)$ by Axiom \ref{ax:primed-square}, and if $f_2'(T(z))\neq b$ we use Axiom \ref{ax:f1primef2-square} to get $c=f_1'(b)=f_2(a)$. In the latter case, since $\widehat{\varphi}_1(b)\ge 2$ we have $\varphi_1(b)\ge 3$ and so $\varphi_1(T(z))\ge 2$ by the lengths axiom.  Thus by using Axiom \ref{ax:primed-square} or \ref{ax:f1primef2-square} again at $a$, we see that $f_1'(c)$ is defined, which is impossible by Axiom \ref{ax:basic1}.  And in the former case, we have either $c\neq f_1(b)$ or $c\neq f_2(a)$, and in either case a similar argument shows either $f_1'(c)$ or $f_2'(c)$ is defined, again a contradiction to Axiom \ref{ax:basic1}.  It follows that $f_1'(T(z))\neq f_1(T(z))$.

Now, let $t=f_1'(T(z))\neq a$.  By Axiom \ref{ax:basic1}, $s=f_1'(a)=f_1(t)$ is defined.  Then since $t\neq a$, Axiom \ref{ax:f1primef2-square} implies that $t$, $T(z)$, and $f_2(T(z))$ meet in an $\{f_1',f_2\}$ square at a vertex $c=f_1'(f_2(T(z)))$, so $f_1'$ is defined at $b$.  Let $d=f_2(a)$.  By Axiom \ref{ax:primed-square} or \ref{ax:f1primef2-square}, vertices $a,s,d$ form a square with a vertex $v=f_1'(d)\neq T(x)$ (since either $T(x)=f_1'(c)$, if the $1/1'$-string through $b$ is collapsed, or $e_1'(x)=\varnothing$, if it is separated).  The same argument shows there is a square at $d,v,T(y)$ with bottom vertex $p=f_1'(T(y))$.  Note that whether we used \ref{ax:primed-square} or \ref{ax:f1primef2-square} to form both squares, we have $f_2(t)=p$, for if we used primed squares, the length-two $f_2'$-string $s\to v\to p$ must be collapsed. Thus $p\neq f_2(T(x))$. 

Notice that the same diagram can be drawn in $G$ to a point $p_0$ corresponding to $p$, and in $G$ we can further conclude that $p_0\neq f_2(x)=w=f_1(y)$.  Thus $f_1'(y)\neq f_1(y)$ in $G$, so the $1$-string through $y$ is separated.  Hence the $1$-string through $T(y)$ is separated as well.  It follows that $\vphi_1(w)\ge 1$, and so $\vphi_1(x)=\vphi_1(T(x))\ge 1$ as well since $\Delta'_w=(0,0)$.  Thus $\vphi_1(b)\ge 3$, and in particular, $\vphi_1(c)\ge 2$.  (See Figure \ref{fig:special-stembridge-case} for an illustration of the vertices $s,v,p$ as well as the extended $1/1'$ string through $b$.)

Recall that $f_2(t)=c$ from our construction above.  Then $t\to c$ and $t\to s$ are two solid edges.  
By Lemma \ref{lem:copying-lengths-along-f1'}, $\Delta_t=\Delta_z=(0,0)$, so arrows $c\leftarrow t \rightarrow s$ meet in an octagon or half-primed octagon pattern, say with bottom element $q$.  It follows that the $1/1'$-chain from $T(y)$ connects up at $q$, that is, $f_1(p)=q$.  By one more application of Axiom \ref{ax:f1primef2-square} at $T(x)$ we see that $f_2(T(x))=f_1(T(y))$ as desired.

\begin{figure}
\begin{center}
\includegraphics[scale=0.75]{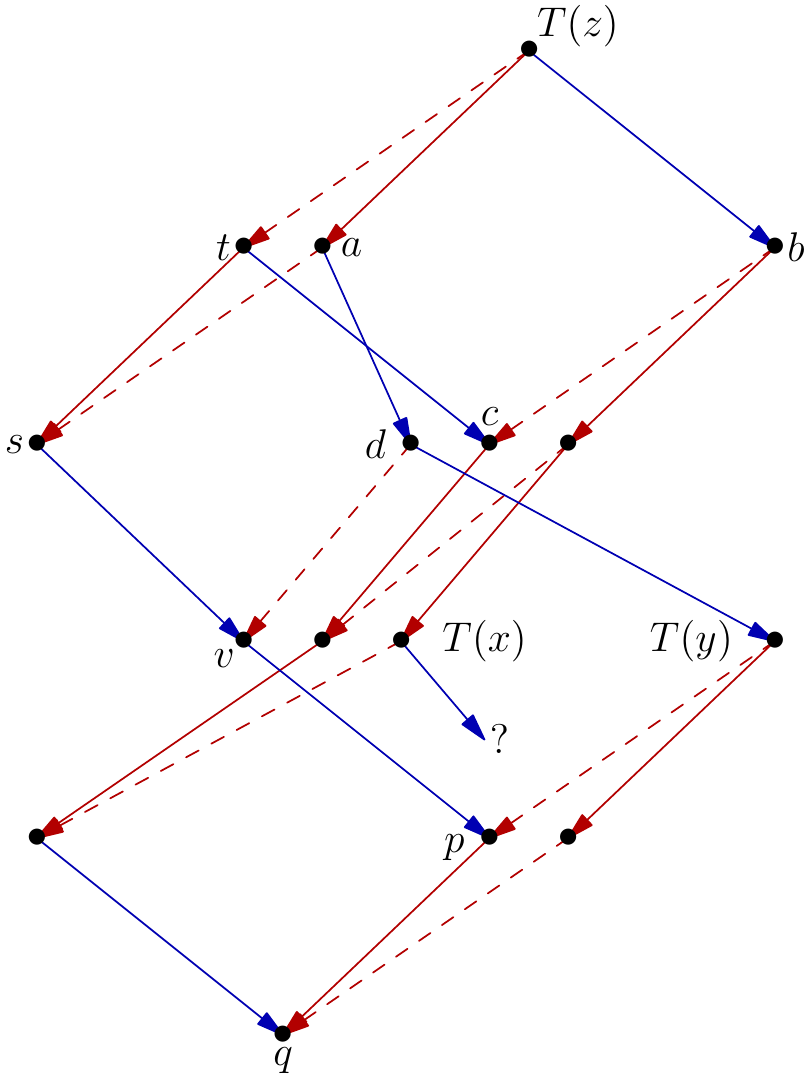}
\end{center}
\caption{\label{fig:special-stembridge-case} Illustration of the proof in Step 5 involving Axiom \ref{ax:stembridge}.}
\end{figure}

\textbf{Step 6: The case $\{f_1,f_2\}$ for $\Delta'_w=(0,0)$, subcase \ref{ax:dual-half-primed-stembridge}.}  Next, consider the case $\Delta_w = (0,0)$ where Axiom \ref{ax:dual-half-primed-stembridge} applies, but where we are not in Step 5. We use a different method from the one used Steps 1-5 here, since there is no axiom to yield a relation $f_2f_1^2f_2'(z)= w = f_1f_2'f_2f_1(z).$

Note that by the induction hypothesis, since $f_1(y)=w$ and $f_2(x)=w$ we know that both $f_2(T(x))$ and $f_1(T(y))$ are defined, and we define $\tilde{w} = f_1(T(y))$.  We wish to show that $f_2(T(x))=\tilde{w}$ as well.

If $e_2'(w)$ is defined, we are done by the steps above. Also, if $e_2'(\tilde w)$ is defined, so is $e_2'(w)$ by reversing the roles of $G$ and $G'$. So we can assume that $e_2'(\tilde w)$ is not defined. 

We now show that $e_2(\tilde w)$ is defined.  Since \ref{ax:dual-half-primed-stembridge} holds at $w$ but not \ref{ax:dual-stembridge}, we have that $\vepshat_2(y)\le 1$ (by the negation of the condition for Axiom \ref{ax:dual-stembridge}).  So in fact $\vepshat_2(y)=1$, and so the $2/2'$-string through $y$ (and hence $T(y)$) is separated and $\veps_2(y) = \veps_2(T(y)) = 2$.  Hence $\veps_2(\tilde w)\ge 1$, by the lengths axiom.  Since $e_2'(\tilde w)$ is not defined, we see that $e_2(\tilde w)$ is defined.

Now, since $e_2(\tilde w)$ is defined and $e_2'(\tilde w)$ is not, then one of Axioms \ref{ax:dual-half-primed-square}-\ref{ax:dual-stembridge} holds at $\tilde{w}$ in $G'$.  If this axiom is not \ref{ax:dual-half-primed-stembridge}, then applying $T^{-1}$ we see that the same dual axiom holds in $G$ at $w$, connecting through $x$ and $y$ by our above arguments, and we have a contradiction.  Thus Axiom \ref{ax:dual-half-primed-stembridge} holds at $\tilde w$, and by following the arrows upwards from $T(y)$ and then downwards, it must pass through $T(x)$ as well.  Thus $f_1(T(x))=\tilde w$, as desired.

 \textbf{Step 7: The case $\{f_1',f_2\}$.} Finally, suppose the two arrows from $y$ and $x$ are $\{f_1', f_2\}$. We may assume that no $f_1$ or $f_2'$ arrow points at $w$: otherwise, by the cases already established, the (isomorphic image of the) additional arrow agrees with both (the images of) $f_1'$ and $f_2$, hence the $f_1'$ and $f_2$ agree with each other. Likewise if either or both arrows are collapsed, we are done by the previous cases. Note that both strings through $w$ are therefore separated, so $e_1'(w) = \varnothing$. Next, if $\widehat{\varphi}_2(w) > 0$, Axiom \ref{ax:dual-f1primef2-square} yields $e_2'(e_1(w))= t = e_1(e_2'(w))$ in $G$. We can't have $f_2'(t) = f_2(t)$ since Axiom \ref{ax:primed-square} would imply $w = f_2'(x) = f_2(x)$, so Axiom \ref{ax:f1primef2-square} applies at $t$, and therefore at $T(t) \in G'$, completing the square in $G'$.

We are left with the case $\widehat{\varphi}_2(w) = \veps'_2(w) = \vepshat_1(w) = \varphi'_1(w) = 0$. By Lemma \ref{lem:excluded-lengths} (excluded lengths), no such $w$ exists, so we're done.

 \textbf{Step 8: Showing $T$ preserves string lengths.}  We have constructed the map of graphs $T : G_{\leq r} \to G'_{\leq r}$. It remains to show that $T$ preserves all string lengths. Let $w \in G_r$. There is some edge $z \xrightarrow{i \text{ or } i'} w$, where $z \in G_{r-1}$. Then the $i$-lengths of $w$ are determined by those of $z$. By induction, $z$ and $T(z)$ have all the same string lengths, so we see that $T(w)$ has the same $i$ lengths as $w$.  Moreover, by Axiom \ref{ax:basic2}, if $j$ is such that $|i-j| > 1$ then the $j$-lengths at $w$ are the same as those at $z$, and so these values are determined as well.

Thus we are done unless all edges pointing at $w$ have the same numerical label, say $i=1$. 
For simplicity say there is a $1$ or $1'$ arrow into $w$ but no $2$ or $2'$ arrow.  In this case $w$ is the top element of its $2/2'$-string, so $\vepshat_2(w) = \veps'_2(w) = 0$
, and likewise for $T(w)$.  Moreover, since $w$ is at the top of its string, we can tell whether $\varphi'_2(w)=1$ or $\varphi'_2(w)=\varphi_2(w)$ by examining the weight of $w$, by Axiom \ref{ax:basic1}.  It thus suffices to show $\varphi_2(w) = \varphi_2(T(w))$. 
We apply Axiom \ref{ax:length}: either $z$ is the top of its own $f_2/f_2'$ string and $\varphi_2(w) = \varphi_2(z) + 1$; or $z$ is one step below the top of its $f_2/f_2'$ string (that is, $\veps_2(z) = 1$) and $\varphi_2(w) = \varphi_2(z)$. Either way, the same condition ($\veps_2 = 0$ or $1$) holds for $T(z)$, so we deduce $\varphi_2(w) = \varphi_2(T(w))$, as desired.
\end{proof}


\section{Shifted tableaux, operators, and the basic axioms} \label{sec:tableaux}

We now recall some additional definitions and results on shifted tableaux and the operators defined in \cite{GLP}, in order to prove that the stated axioms hold for these crystal-like structures.  We first recall the definition of standardization for shifted tableaux.

\begin{figure}[b]
  \begin{center}
    \includegraphics{shifted-semistandard.pdf}  \hspace{3cm} \includegraphics{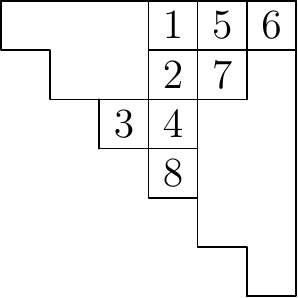}
  \end{center}
  \caption{\label{fig:standardization} A semistandard shifted tableau and its standardization.}
\end{figure}
  
  \begin{definition}
  The \textbf{standardization} of a word $w$, or tableau $T$, is the word $\mathrm{std}(w)$, or tableau $\mathrm{std}(T)$ formed by replacing the letters in order with $1,2,\ldots,n$ from least to greatest, breaking ties by reading order for unprimed entries and by reverse reading order for primed entries.  (see Figure \ref{fig:standardization}).
  \end{definition}

Recall that we normally require that a word or tableau is in \textbf{canonical form}, that is, the leftmost $i$ or $i'$ in reading order is unprimed for all $i$.  Notice, however, that priming the first $i$ in such a word or tableau does not change its standardization.  We think of a word (or tableau) as an equivalence class of strings (or generalized tableaux) consisting of all the words (or tableaux) formed by either priming or unpriming the first $i$ or $i'$ for each $i$.  Such a way of priming the entries is called a \textbf{representative} of the word.  

\subsection{Raising and lowering operators on shifted tableaux}
  
As with ordinary tableaux, the weight raising and lowering operators can be defined in terms of the reading word.  In \cite{GLP}, four such classes of operators are defined, namely $E'_i,F'_i,E_i,F_i$.  While the original definitions are recursive in terms of shifted jeu de taquin, we use here the main results of \cite{GLP}, namely the local rules for the operators, as their definitions.

Let $\alpha_i=(0,0,\ldots,0,1,-1,0,\ldots,0)$ be the tuple consisting of a $1$ in the $i$-th position and $-1$ in the $i+1$-st position, that is, the $i$-th standard type A simple root.  The primed operators $E'_i$ and $F'_i$ may be defined in terms of standardization as follows.

\begin{definition} \label{def:primed-operators}
We define $E_i'(w)$ to be the unique word such that
\[\mathrm{std}(E_i'(w)) = \mathrm{std}(w) \hspace{0.5cm}\text{ and }\hspace{0.5cm} \mathrm{wt}(E_i'(w)) = \mathrm{wt}(w) + \alpha_i,\]
if such a word exists; otherwise, $E_i'(w) = \varnothing$. We define $F_i'(w)$ analogously using $-\alpha_i$.
\end{definition}

The following proposition gives a more explicit way of computing the primed operators \cite{GLP}:

\begin{proposition}\label{prop:explicit-definitions-primed}
To compute $F_i'(w)$ consider all representatives of $w$. If all representatives have the property that the last $i$ is left of the last $(i+1)'$ then $F'(w) = \varnothing$.  If there exists a representative such that the last $i$ is right of the last $(i+1)'$ then $F'(w)$ is obtained, using this representative, by changing the last $i$ to $(i+1)'$.

The word $E'(w)$ is defined similarly with the roles of $i$ and $(i+1)'$ reversed: if the last $(i+1)'$ is right of the last $i$ in some representative, change it to a $i$ (in that representative). Otherwise $E'(w) = \varnothing$.
\end{proposition}

To define the unprimed operators $E_i$ and $F_i$, we first require the notion of the $i$-th \textbf{lattice walk} of a word $w$.  For $i=1$, we consider the subword of $w$ consisting of the letters $1,1',2,2'$.  The walk then begins at the origin and converts each letter of this subword to a unit step in a specific cardinal direction, as defined and illustrated in Figure \ref{fig:lattice-walk-intro}.  The $i$-th walk is obtained in the same way with the letters $i,i',i+1,(i+1)'$ occurring in $w$, where the cardinal directions are assigned as before but replacing $1$ with $i$ and $2$ with $i+1$.

\begin{figure}
\begin{center}
\begin{tabular}{cc}
\setlength{\unitlength}{3em}
\ \ \begin{picture}(5,3)(0,0)
\multiput(0,0)(0,0.2){16}{\line(0,1){0.1}}
\multiput(0,0)(0.2,0){17}{\line(1,0){0.1}}
\put(0,1.2){\circle*{0.1}}
\put(0,1.2){\vector(0,1){.7}\vector(1,0){.7}}
\put(-0.265,1.95){\small{$2$,$2'$}}
\put(0.4,1.3){\small{$1$,$1'$}}%
\put(1.4,0){\circle*{0.1}}
\put(1.4,0){\vector(0,1){.7}\vector(1,0){.7}}
\put(1.5,0.6){\small{2,2'}}
\put(2.0,0.1){\small{1,1'}}%
\put(2.8,2.2){\circle*{0.1}}
\put(2.8,2.2){\vector(0,1){.7}\vector(1,0){.7}}
\put(2.8,2.2){\vector(0,-1){.7}}
\put(2.8,2.2){\vector(-1,0){.7}}
\put(2.7,2.95){\small{$2$}}
\put(3.5,2.1){\small{$1'$}}
\put(2.7,1.2){\small{$1$}}
\put(1.85,2.1){\small{$2'$}}
\end{picture}
&
\ \
\begin{picture}(3,3)(0,0)
\setlength{\unitlength}{3em}
\stepnorth{2}{%
\stepeast{1}{%
\stepeast{1'}{%
\stepsouthshiftW{1}{%
\stepnorthA{2'}{%
\stepnorthA{2}{%
\stepwestshiftS{2'}{%
\stepeastA{1'}{%
\stepeastA{1'}{%
}}}}}}}}}
\put(0,0){\circle*{0.15}}
\put(3,2){\circle{0.15}}
\multiput(0,0)(0,0.2){15}{\line(0,1){0.1}}
\multiput(0,-0.02)(0.2,0){15}{\line(1,0){0.1}}
\end{picture}
\end{tabular}
\end{center}
\caption{The {\bf lattice walk} of a word $w \in \{1', 1, 2', 2\}^n$. In the interior of the first quadrant, each letter corresponds to a cardinal direction. Along the axes, primed and unprimed letters behave the same way. \textbf{Right:} The walk for $w=211'12'22'1'1'$ ends at the point $(3,2)$.\label{fig:lattice-walk-intro}} 
\end{figure}
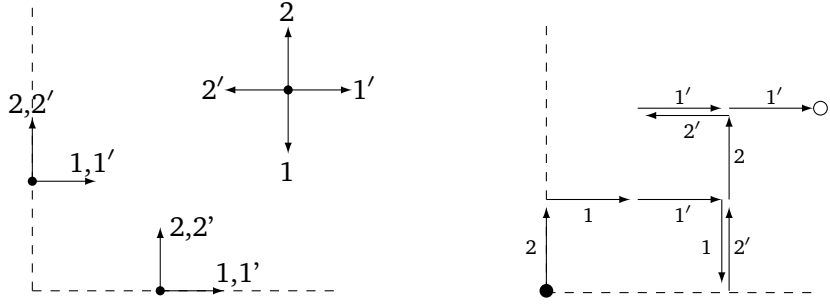

In the remainder of this section, for simplicity we assume $w$ is a word in the alphabet $\{1,1',2,2'\}$ and define only $E_1$ and $F_1$.  The definitions of $E_i$ and $F_i$ are then defined analogously on the $\{i,i',i+1,(i+1)'\}$-subword of a word by replacing $1,1',2,2'$ with $i,i',i+1,(i+1)'$ respectively in the definitions.
\begin{figure}
\begin{center}
\begin{tabular}{|c|c|c|c|c|}
\hline
\multirow{2}{*}{Type} 
& \multicolumn{3}{c|}{Conditions} & \multirow{2}{*}{Transformation}  \\
     & \multicolumn{1}{c}{Substring}&  \multicolumn{1}{c}{Steps} & Location &  
\\\hline
\hline
 \multirow{2}{*}{1F} & 
\multirow{2}{*}{$u = 1(1')^*2'$} &
\east{1} ~ \east{1'} ~ \north{2'} & 
$y=0$ & 
\multirow{2}{*}{$u \to 2'(1')^*2$} \\[.5ex]\cline{3-4}
& &  
\south{1} ~ \east{1'} ~ \north{2'} & 
$y=1$, $x \geq 1$ & 
\\[.5ex]\hline
\multirow{2}{*}{2F} &
\multirow{2}{*}{$u = 1(2)^*1'$} &
\east{1} ~ \north{2} ~ \east{1'} & 
$x = 0$ & 
\multirow{2}{*}{$u \to 2'(2)^*1$}  \\[.5ex]\cline{3-4}
&& 
\south{1} ~ \north{2} ~ \east{1'} &
$x = 1$, $y \geq 1$ & 
\\[.5ex]\hline
 3F & $u = 1$ & 
\east{1} & 
$y = 0$ & 
$u \to 2$ 
\\\hline
 4F & 
$u = 1'$ & 
\east{1'} & 
$x  = 0$ & 
$u \to 2'$ 
\\\hline
\multirow{2}{*}{5F} & $u = 1$ 
& \south{1}
& \multirow{2}{*}{$x=1$, $y \geq 1$} 
&
\multirow{2}{*}{undefined} \\[.5ex]\cline{2-3}
& $u=2'$ & \west{2'} &&
\\\hline
\end{tabular}
\end{center}
\vspace{0.5cm}

\begin{center}
\begin{tabular}{|c|c|c|c|c|}
\hline
\multirow{2}{*}{Type} 
& \multicolumn{3}{c|}{Conditions} & \multirow{2}{*}{Transformation}  \\
     & \multicolumn{1}{c}{Substring}&  \multicolumn{1}{c}{Steps} & Location &  
\\\hline
\hline
 \multirow{2}{*}{1E} & 
\multirow{2}{*}{$u = 2'(2)^*1$} &
\north{2'} ~ \north{2} ~ \east{1} & 
$x=0$ & 
\multirow{2}{*}{$u \to 1(2)^*1'$} \\[.5ex]\cline{3-4}
& &  
\west{2'} ~ \north{2} ~ \east{1} & 
$x=1$, $y \geq 1$ & 
\\[.5ex]\hline
\multirow{2}{*}{2E} &
\multirow{2}{*}{$u = 2'(1')^*2$} &
\north{2'} ~ \east{1'} ~ \north{2} & 
$y = 0$ & 
\multirow{2}{*}{$u \to 1(1')^*2'$}  \\[.5ex]\cline{3-4}
&& 
\west{2'} ~ \east{1'} ~ \north{2} &
$y = 1$, $x \geq 1$ & 
\\[.5ex]\hline
 3E & $u = 2'$ & 
\north{2'} & 
$x = 0$ & 
$u \to 1'$ 
\\[.5ex]\hline
 4E & 
$u = 2$ & 
\north{2} & 
$y  = 0$ & 
$u \to 1$ 
\\[.5ex]\hline
\multirow{2}{*}{5E} & $u = 1$ 
& \south{1}
& \multirow{2}{*}{$y=1$, $x \geq 1$} 
&
\multirow{2}{*}{undefined} \\[.5ex]\cline{2-3}
& $u=2'$ & \west{2'} &&
\\\hline
\end{tabular}
\end{center}

\caption{\label{fig:criticals} Above, the table of $F_1$-critical substrings and their transformations.  Below, the table of $E_1$-critical substrings and their transformations.  Here $a(b)^*c$ means any string of the form $abb \dots bc$, including $ac$, $abc$, $abbc$, etc.}
\end{figure}

\begin{definition}
If $w$ is a word (in $\{1,1',2,2'\}$) and $u = w_k w_{k+1} \dots w_l$ is a substring of some representative of $w$, then $u$ is a {\bf substring} of the word $w$.  The coordinates $(x,y)$ of the point in the $1$-walk just before the start of $u$ is the \defn{location} of $u$.
\end{definition}

\begin{definition} 
A substring $u$ is an \defn{$F_1$-critical substring} if certain conditions on $u$ and its location are met. There are five types of $F_1$-critical substrings. Each row of the first table in Figure \ref{fig:criticals} describes one type, the directions of its arrows in the walk, and a transformation that can be performed on that type.

For example, the first row indicates that $u$ is a critical substring of type 1F if $u$ is of the form $1(1')^*2'$ and the location $(x,y)$ satisfies $y=0$ or $y=1,\, x \geq 1$.
\end{definition}

The {\bf final} $F_1$-critical substring $u$ of $w$ is defined as follows: we take $u$ with the highest possible starting index, and take the longest in the case of a tie. If there is still a tie (from different representatives of $w$), we take any such $u$. 

\begin{definition}\label{def:F}
To define $F_1(w)$, fix a representative $v$ of $w$ containing the final $F_1$-critical substring $u$, and transform $u$ (in $v$) according to its type. The resulting word (i.e., after canonicalizing) does not depend on the choices, if any, of $u$ or $v$. If the type is 5F, or if $w$ has no $F_1$-critical substrings, then $F_1(w)$ is undefined; we write $F(w) = \varnothing$.

The definition of $E_1(w)$ is analogous using $E_1$-critical substrings.
\end{definition}

\begin{remark}
As an operator on words, the symmetry involution $\eta$ (see Section \ref{subsec:eta} ) simply replaces any letter $i$ with $(n+1-i)'$ and vice versa, and then unprimes the leftmost $i$ or $i'$ if necessary for each $i$.  For instance, if $n=3$, we have $$\eta(33'122'132)=113223'1'2'.$$ 
It is easy to see from Figure \ref{fig:criticals} that $\eta$ interchanges $E_i$ with $F_{n+1-i}$ for all $i$. Since the effect of $\eta$ on $\std(w)$ is to reverse the standardization order, $\eta$ also interchanges $E_i'$ with $F_{n+1-i}'$. For a tableau $T \in \ShST(\lambda/\mu,n)$ with reading word $w$, $\eta(w)$ naturally forms a ``reflected'' tableau:
\[
\begin{ytableau}
\none & 1' & 2 & 3 \\ 1 & 1
\end{ytableau}
\qquad \leadsto \qquad
\raisebox{2ex}{\begin{ytableau}
\none&1\\ 
\none&2\\
3' & 3 \\
3
\end{ytableau}}
\]
We therefore define $\eta(T) \in \ShST(\lambda/\mu,n)$ as the unique tableau dual equivalent to $T$, but slide equivalent to the reflected tableau (under shifted jeu de taquin). Since the operators $F_i, F_i'$ are coplactic, it is still true that $\eta$ interchanges $E_i$ with $F_{n+1-i}$ and $E_i'$ with $F_{n+1-i}'$, for all $i$. (For tableaux of shifted straight shape, $\eta$ reduces to Sch\"{u}tzenberger's evacuation, see \cite{Schutzenberger}.)
\end{remark}


\subsection{Axioms \ref{ax:basic1}, \ref{ax:basic2}, \ref{ax:length}, \ref{ax:Kashiwara-B}}

We briefly discuss the basic axioms. Axiom \ref{ax:basic1} holds because the operators $F$ and $F'$ are coplactic (and \ref{ax:basic1} clearly holds for rectified tableaux). Axiom \ref{ax:basic2} is immediate from the fact that the operators $F_i$ and $F_j$ act on disjoint subwords when $|i-j|>1$. Axiom \ref{ax:Kashiwara-B} is shown in Proposition 6.9 of \cite{GLP}. We are left with showing the Lengths Axiom \ref{ax:length}.

\begin{proposition}\label{prop:length}
Let $z = f_{i\pm 1}(w)$ or $f'_{i\pm 1}(w)$. Then
\[(\varepsilon_i(w)-\varepsilon_i(z), \varphi_i(w)-\varphi_i(z)) = (0,1) \text{ or } (-1,0).\]
\end{proposition}
\begin{proof}
By applying $\eta$ and exchanging $z$ and $w$, we may assume the edge has numerical value $i-1$. We may then take $i=2$. By coplacticity, we may assume $w$ is the reading word of a rectified shifted tableau.

We compare the 2,3-subwords of $w$ and of $z = F_1(w)$ or $F_1'(w)$. The effect of $F_1'$ is to insert a $2'$ into this subword, while the effect of $F_1$ is to insert a $2$. Either way, $z$ and $w$ have the same suffix after the inserted letter, but the tail of the lattice walk for $z$ is initially offset from that of $w$, either by one step downwards or one step rightwards. By the Bounded Error Lemma (Proposition 4.9 of \cite{GLP}), the endpoint of the lattice walk for $z$ is also offset by one step either downwards or rightwards (not necessarily the same as the initial offset). These outcomes give the two cases of the statement.
\end{proof}

\section{Establishing the merge axioms for shifted tableaux}

We now prove that the merge axioms hold for the coplactic crystal-like structures on shifted tableaux.  Note that a simple application of $\eta$ proves if an axiom holds, then its dual axiom holds.  So it suffices to prove one of the axiom or its dual.  We establish one axiom/dual pair in each subsection below.

We prove the axioms only for $i=1$ and $i+1=2$, since the general case is identical.  We can therefore further restrict to tableaux having entries from the alphabet $\{1',1,2',2,3',3\}$, and since all the operators are coplactic, we can assume the tableaux have straight shape.  In all of what follows below, we therefore assume $T$ is a straight shape semistandard shifted tableau in the letters $\{1',1,2',2,3',3\}$.  We will furthermore assume that $T$ is in canonical form, so that whenever we discuss which entries occur in $T$, we are specifically discussing its canonical representative.
 
We begin with a lemma that will simplify the last four proofs below.

\begin{lemma}\label{lem:Delta-tableaux}
  Suppose $F_1(T)$ and $F_2(T)$ are defined. We have $\Delta\veps_1=\veps_1(T)-\veps_1(F_2(T))=0$ if and only if the final $F_2$-critical substring of $T$ starts in the second row (rather than the top row).
  
  We have $\Delta\veps_2 = \veps_2(T)-\veps_2(F_1(T))=0$ if and only if the $x$-coordinate of the endpoint of the $2,3$-walk changes after inserting a $2$ into the top row of $T$. 
\end{lemma}

\begin{proof}
 Note that by Axiom \ref{ax:length} (Proposition \ref{prop:length}), $\varepsilon_1(T)-\varepsilon_1(F_2(T))=0$ if and only if $\varphi_1(T)-\varphi_1(F_2(T))=-1$ (and otherwise these differences are $1$ and $0$ respectively).  The value of $\varphi_1(T)$ is the difference between the number of $1$s in the top row and the number of $2$s in the second row of $T$.  If applying $F_2$ changes an $F_2$-critical substring starting in the second row, then one of the $2$ entries in the second row changes to a $3$ or $3'$, which does change $\varphi_1$; otherwise, if it starts in the top row, $\varphi_1$ does not change.  The first statement follows.
  
The second statement is a reformulation of the lengths axiom \ref{ax:length} (note that $\varphi_2$ is the $x$-coordinate of the endpoint of the $2,3$-walk).
\end{proof}

\subsection{Axiom \ref{ax:primed-square}: Primed-square}

\begin{proposition}
  If $F_1'(T)$ and $F_2'(T)$ are both defined, then $$F_1'(F_2'(T))=F_2'(F_1'(T))\neq \varnothing.$$
\end{proposition}

\begin{proof}
By the definition of the primed operators, it suffices to show that $F'_1(F'_2(T))$ is defined, since it is then the (unique) word of the desired weight and standardization, and hence agrees with $F'_2(F'_1(T))$.

Note that $\wt_1(F'_2(T)) = \wt_1(T) > 0$.  For tableaux containing a $1$, $F'_1$ is defined unless the second row is nonempty and contains a $2$, \emph{and} the first row contains a $2'$. Applying $F'_2$ changes a $2$ to a $3'$ and possibly deletes a prime from a $2$; regardless it cannot introduce those conditions, so $F_1'$ will still be defined.
\end{proof}

\subsection{Axiom \ref{ax:half-solid-square}: Half-solid Square}

\begin{proposition}\label{prop:tableau-half-solid-square}
  Suppose both $F_1'(T)$ and $F_2'(T)$ are defined.  Then $$S:=F_1(F_2'(T))=F_2(F_1'(T))\neq \varnothing\hspace{0.3cm}\text{ and }\hspace{0.3cm}F_1'(F_2'(T))\neq S$$ if and only if $\Delta=(0,0)$ and $\varphi_2(T)=1$ and $\widehat{\varphi}_2(T)=0$.
\end{proposition}

\begin{proof}
  Since both $F_1'(T)$ and $F_2'(T)$ are defined, we know that there is at least one $1$ and at least one $2$ in $T$ and that there is no $2'$ in the top row (in canonical form).  Moreover, the last $2$ is to the right of the last $3'$, even in canonical form.
  
  We will show that both statements are equivalent to the statement that $T$ has rows $$(1^{c}3^{d},2^{a+1}3^{b},3^a)$$ from top to bottom for some $a\ge 1$ and $b,c,d\ge 0$ such that $T$ is semistandard. 
  
  It is easy to check that if $T$ has this form, then $F_1(F_2'(T))=F_2(F_1'(T))$, where the application of $F_2$ on the right is a type 2F final critical substring.  It is also easy to see that the square is not simultaneously a primed square.  This is illustrated in the following diagram.
\begin{center}
\begin{tikzpicture}[x=0.6cm,y=0.6cm]
  \node (top) at (0,0) {\scriptsize $\begin{ytableau} 1 & 1 & 1 & 1 & 1 & 1 & 3 & 3 \\ \none & 2 & 2 & 2 & 3 \\ \none & \none & 3 & 3 \end{ytableau}$};
  \node (left) at (-4,-4) {\scriptsize $\begin{ytableau} 1 & 1 & 1 & 1 & 1 & 2' & 3 & 3 \\ \none & 2 & 2 & 2 & 3 \\ \none & \none & 3 & 3 \end{ytableau}$};
  \node (right) at (4,-4) {\scriptsize $\begin{ytableau} 1 & 1 & 1 & 1 & 1 & 1 & 3 & 3 \\ \none & 2 & 2 & 3' & 3 \\ \none & \none & 3 & 3 \end{ytableau}$}; 
  \node (bot) at (0,-8) {\scriptsize $\begin{ytableau} 1 & 1 & 1 & 1 & 1 & 2 & 3 & 3 \\ \none & 2 & 2 & 3' & 3 \\ \none & \none & 3 & 3 \end{ytableau}$};
  
  \draw [red, dashed, ->, thick] (top)--(left);
  \draw [blue, dashed, ->, thick] (top)--(right);
  \draw [red, ->, thick] (right)--(bot);
  \draw [blue, ->, thick] (left)--(bot);
\end{tikzpicture}
\end{center}
  Furthermore, by analyzing the $2,3$-lattice walk of such a tableau we see that it ends at $x=1$ so $\vphi_2(T)=1$. Since the final $F_2$-critical substring is type 5F we have $F_2(T)=\varnothing$ and hence $\vphihat_2(T)=0$.  Finally, we have $\veps_1(T)=\veps_1(F_2'(T))=0$ since we cannot apply $E_1$, and we have $\veps_2(T)=\veps_2(F_1'(T))$ since the $y$-coordinate of the $2,3$-walk does not change upon inserting a $2'$ before a final run of $3$s.  Thus if $T$ has the above form it satisfies the numerical conditions $\Delta=(0,0)$, $\varphi_2=1$, $\widehat{\varphi}_2=0$.
  
  We now show that if $F_1(F_2'(T))=F_2(F_1'(T))$ and the square is not simultaneously a primed square, then $T$ has the desired form above.  Since its $F_1(F_2'(T))\neq F_1'(F_2'(T))$, we know that $F_2'(T)$ has at least $2$ rows, with a $2$ in the second row, and hence $T$ itself has at least $2$ rows with a $2$ in the second row as well. 

Now, assume for contradiction there is also a $2$ in the top row of $T$. Consider the square $s$ containing the rightmost $2$ in the top row of $T$. Then $s$ is a $2$ in $F_1'(T)$, but a $3'$ in $F_1(F_2'(T)) = F_2(F_1'(T))$, so $s$ must be the first entry of the final $F_2$-critical string of $F_1'(T)$. In particular, $F_2(F_1'(T))$ still contains a $2'$ (which occurs before $s$), while $F_1(F_2'(T))$ does not, a contradiction. Next, we see that the canonical form of $T$ does not contain a $3'$, since $F_2'(T)$ is defined. We also see that the third row of the tableau is nonempty, since otherwise both $F_2$ arrows would be collapsed and the square would be simultaneously a primed square.

  
  Let $a$ be the number of $3$s in the bottom row and assume for contradiction that the number of $2$s in the second row is greater than $a+1$.  Then the application of $F_2$ to $F_1'(T)$ is a type 3F critical substring (rather than 2F) in the second row which simply changes the last $2$ to a $3$, and therefore leaves the $2'$ in the top row.  Thus it does not match $F_1(F_2'(T))$.  It follows that the number of $2$s is indeed $a+1$.  Thus $T$ has the desired form.
  
  Finally we show that if $T$ satisfies the numerical conditions $\Delta=(0,0)$ and $\varphi_2(T)=1$ and $\widehat{\varphi}_2(T)=0$ then $T$ has the form above.  Indeed, assuming these conditions hold we have that $\veps_1(T)=\veps_1(F_2'(T))$ by the definition of $\Delta$.  This means that after changing the last $2$ to a $3'$, the total number of $2/2'$s in the top row stays the same, and so in fact there must be none and $\veps_1(T)=0$. We also see that there are no $3'$ entries (since $F_2'(T)$ is defined), but there is a $2$ in the second row, and the third row is nonempty (otherwise $F_2(T) = F_2'(T)$, contradicting $\widehat{\varphi}_2(T) = 0$). Finally, since $\varphi_2(T)=1$ but $\widehat{\varphi}_2(T)=0$, the number of $2$s in the second row is exactly one more than the number of $3$s in the third row.
\end{proof}

\subsection{Axiom \ref{ax:dual-f1primef2-square}: Dual Square for $\{e_2',e_1\}$}

\begin{definition}
  A shifted straight shape tableau $T$ on the alphabet $\{1',1,2',2,3',3\}$ is \textbf{special} if (a) the second row of $T$ is nonempty, (b) there is exactly one $2$ or $2'$ in $T$ and it occurs in the top row (and is therefore a $2$), and (c) there is no $3'$ in the top row.
\end{definition}

\begin{example}
  The following is a special tableau: 
  \begin{center}\small
  \begin{ytableau} 
    1 & 1 & 1 & 1 & 2 & 3 & 3 \\
    \none & 3 & 3 
  \end{ytableau}
  \end{center}
\end{example}

In particular, a special tableau must have at most two rows by semistandardness.  We now show that our operators satisfy Axiom \ref{ax:dual-f1primef2-square}.

\begin{proposition}\label{prop:E1E2prime}
  Suppose $T$ is not special, and both $E_2'(T)$ and $E_1(T)$ are defined.  Then $E_1(E_2'(T))$ and $E_2'(E_1(T))$ are defined and are equal to each other: 
$$E_1(E_2'(T))=E_2'(E_1(T))$$
In particular, since special tableaux have the property that $E_1'(T)=E_1(T)$ and $E_2'(T)=E_2(T)$, the above relation holds whenever $E_2'(T)$ and $E_1(T)$ are defined and either $E_1'(T)\neq E_1(T)$ or $E_2'(T)\neq E_2(T)$.
\end{proposition}

We first state an alternate description for $E'_2$ that follows immediately from Proposition \ref{prop:explicit-definitions-primed}.

\begin{proposition}\label{prop:simpler-E2prime}
  To compute $E'_2(w)$, let $x$ be the last $3'$ in $w$ (including the first $3$).  If there is only one $2$ in $w$, say $y=2$, and $x$ is to the left of $y$, then $E'(w)$ is obtained by changing $x$ to $2$ and $y$ to $2'$.  If simply changing $x$ to $2$ does not affect $\mathrm{std}(w)$, then $E'(w)$ is formed by making that change and canonicalizing.  Otherwise $E'(w)=\varnothing$.
\end{proposition}

We now prove Proposition \ref{prop:E1E2prime}.

\begin{proof}
   Let $T$ be a straight shape semistandard shifted tableau for which $E_2'(T)$ and $E_1(T)$ are defined. Let $x$ be the last $3'$ in reading order in $T$ and $y$ the first $2$.
   
   First suppose $x$ precedes $y$ in reading order and $y$ is the only $2$ in $T$.  
Then $y$ is in the first row (since $E_1$ is defined) and $T$ has at most two rows. By semistandardness $x$ must be in the second row, so $T$ is special.

Now suppose $y$ precedes $x$ in reading order. Since $E'_2(T)$ is defined, it follows from Proposition \ref{prop:simpler-E2prime} that changing $x$ from $3'$ to $2$ does not affect the standardization of $T$, and $E'_2(T)$ is formed by making this change. On the other hand, $E_1(T)$ is formed by changing the first $2$ or $2'2$ in the top row of $T$ to $1$ or $12'$ respectively.  Thus there is a $2$ in the top row of $T$, and so if $x$ were in a lower row changing it to $2$ would affect the standardization.  So $x$ is also in the top row, and the operations commute.
 \end{proof}

\subsection{Axiom \ref{ax:f1f2prime-square}: Square for $\{f_1,f_2'\}$}

\begin{proposition} 
  Suppose both $F_1(T)$ and $F_2'(T)$ are defined.  Then $$F_1(F_2'(T))=F_2'(F_1(T))\neq \varnothing$$ if and only if $\vepshat_1(T)>0$.
\end{proposition}

\begin{proof}
  Assume $\vepshat_1(T)>0$.  Then there is at least one $2$ in the top row of $T$.  Since $F_2'(T)$ is defined, there is no $3'$ in the top row, and the last $2$ changes to $3'$ in the top row to form $F_2'(T)$.  Then applying $F_1$ changes the last $1$ to $2$ or the last $12'$ to $2'2$ in the top row.  The composition $F_2'(F_1(T))$ simply applies these operations to the top row in the other order, and the square commutes, as shown in the example below.
  
\begin{center}
\begin{tikzpicture}[x=0.6cm,y=0.6cm]
  \node (top) at (0,0) {\scriptsize $\begin{ytableau} 1 & 1 & 1 & 1 & 1 & 2 & 3 \\ \none & 2 & 2 & 2 & 3 & 3 \\ \none & \none & 3 & 3 \end{ytableau}$};
  \node (left) at (-4,-4) {\scriptsize $\begin{ytableau} 1 & 1 & 1 & 1 & 2 & 2 & 3 \\ \none & 2 & 2 & 2 & 3 & 3 \\ \none & \none & 3 & 3 \end{ytableau}$};
  \node (right) at (4,-4) {\scriptsize $\begin{ytableau} 1 & 1 & 1 & 1 & 1 & 3' & 3 \\ \none & 2 & 2 & 2 & 3 & 3 \\ \none & \none & 3 & 3 \end{ytableau}$}; 
  \node (bot) at (0,-8) {\scriptsize $\begin{ytableau} 1 & 1 & 1 & 1 & 2 & 3' & 3 \\ \none & 2 & 2 & 2 & 3 & 3 \\ \none & \none & 3 & 3 \end{ytableau}$};
  
  \draw [red, ->, thick] (top)--(left);
  \draw [blue, dashed, ->, thick] (top)--(right);
  \draw [red, ->, thick] (right)--(bot);
  \draw [blue, dashed, ->, thick] (left)--(bot);
\end{tikzpicture}
\end{center}
  For the other direction, assume $\vepshat_1(T)=0$.  Then there is no $2$ in the top row. Since $F_2'(T)$ is defined, there is a $2$ or $2'$ elsewhere in the tableau, so $T$ has at least two rows. While $F_2'(F_1(T))$ is formed using only changes in the top row as before, the composition $F_1(F_2'(T))$ starts by changing a $2$ in the second row to a $3'$. Thus the compositions are not equal.
\end{proof}


\subsection{Axiom \ref{ax:half-primed-square}: Half-primed square}

\begin{proposition}\label{prop:tableau-half-primed-square}
Suppose both $F_1(T)$ and $F_2(T)$ are defined, and $F_1'(T)$ is not defined.  Then $$F_2'(F_1(T))=F_1'(F_2(T))\neq \varnothing$$ if and only if $\Delta=(1,1)$.
\end{proposition}

\begin{proof}
  Assume $\Delta=(1,1)$.  
 By Lemma \ref{lem:Delta-tableaux}, the final $F_2$-critical substring of $T$ starts in the top row (and is therefore not type 2F by semistandardness), and inserting a $2$ in the top row does not change the $x$-coordinate of the endpoint of the $2,3$-walk of $T$.
 
  Assume for contradiction that the final $F_2$-critical substring is type 1F or 3F.  Then the last $2$ in the top row occurs on the $x$-axis in the $2,3$-walk.  Since the remaining entries after this point in the walk are $3'$ or $3$ (from the top row of $T$), inserting a $2$ moves the $x$-coordinate of the endpoint to the right, a contradiction.
  
  Thus the final $F_2$-critical substring is type 4F, and so there is a $2'$ in the top row not followed by a $2$ or $3'$ (for otherwise there would be a type 5F critical substring later in the word).  So the top row is of the form $1^k 2' 3^j$ for some $k$ and $j$.  Applying $F_1$ transforms this row into $1^{k-1}2'23^j$, and then applying $F_2'$ yields $1^{k-1}2'3'3^j$.  If instead we apply $F_2$ we obtain $1^k3'3^j$ in the top row, which then maps via $F_1'$ to $1^{k-1}2'3'3^j$, as desired.
  
  For the converse, suppose $F_1'(F_2(T))=F_2'(F_1(T))=R$ for some tableau $R$.  Then since this square is not simultaneously a primed square by our assumption that $F_1'(T)=\varnothing$, the tableau $R$ has the dual half-solid square property and so $\Delta'$ at $R$ is $(0,0)$ by Proposition \ref{prop:tableau-half-solid-square}. Tracing these differences in string lengths back around the square, we see that $\Delta$ at $T$ is $(1,1)$ as well.
\end{proof}

\subsection{Axiom \ref{ax:square}: Square}

\begin{proposition}\label{prop:tableau-square}
Suppose both $F_1(T)$ and $F_2(T)$ are defined but $F_1'(T)$ is not defined.  Then $$F_2(F_1(T))=F_1(F_2(T))\neq \varnothing$$ if and only if $\Delta=(0,1)$ or $\Delta=(1,0)$.
\end{proposition}

\begin{proof}
  First suppose $\Delta=(1,0)$. By Lemma \ref{lem:Delta-tableaux}, the final $F_2$-critical substring of $T$ starts in the top row (and is therefore not type 2F by semistandardness), and inserting a $2$ in the top row (i.e. applying $F_1$) increases the $x$-coordinate of the endpoint of the $2,3$-walk of $T$ by 1.
%
%
  Note that since $F_1'(T)=\varnothing$, there is a $2'$ in the top row.
  
  
  Assume for contradiction that the $F_2$-critical substring is the $2'$ in the top row of type 4F, and is not simultaneously type 1F or 3F (which can occur if $T$ has one row).  If there are any $2$ entries after it then since there is no type 1F or 3F substring, the last such $2$ does not reach the $x$-axis, and so it is a type 5F critical substring, a contradiction.  Similarly it is not followed by $3'$.  Thus the top row is of the form $2'3^\ast$, and inserting a $2$ does not change the ending $x$-coordinate, a contradiction. 
  
  It follows that the final $F_2$-critical substring is type 1F or 3F in the top row, and so applying $F_1$, which simply inserts a $2$ into the $2,3$-walk just before the critical substring, does not change the location of the $F_2$-critical substring in the tableau.  Thus $F_1(F_2(T))=F_2(F_1(T))$ as desired.
  
  Now suppose $\Delta=(0,1)$.  By Lemma \ref{lem:Delta-tableaux}, the final $F_2$-critical substring starts in the second row, and the $x$-coordinate of the $2,3$-walk does not change after inserting a $2$ into the top row.
  
  Assume for contradiction that the final $F_2$-critical substring is type 1F and ends in the top row.  Then it is of the form $22'3'$ where the $2$ is the last entry in the second row, the $2'$ is in the top row, and the $3'$ is just after it in the top row.  Note that applying $F_1$ has the effect of inserting a $2$ to make this sequence $22'23'$ instead, and since the $2,3$-walk must be on the $x$-axis just before the new $2$ is inserted, a rightwards arrow is inserted, which moves the entire $3'3^\ast$ string after it one step to the right.  Thus the $x$-coordinate does change, a contradiction.
  
  Next suppose the final $F_2$-critical substring is type 2F, of the form $23^\ast 2'$ where the $2'$ is in the top row and the $23^\ast$ is in the second.  Note that the walk just after the $2'$ has $x$-coordinate $x=2$.  If at this point we also have $y=0$, then inserting a $2$ would shift the end of the walk to the right, so we have $y\ge 1$.  Furthermore, by the same reasoning, we have $y\ge 1$ just after the last $2$ in $T$, and if $y=1$ just after this last $2$ then it is not followed by a $3'$.  Thus inserting the $2$ does not create a later type 1F or 3F critical substring, and since we start at $x=2$ after the 2F it does not create a later 5F, 2F, or 4F.  Thus $F_1(T)$ also has a split type 2F final $F_2$-critical substring of the same form, and it is now clear that $F_2(F_1(T))=F_1(F_2(T))$ in this case.
  
  Finally, suppose the final $F_2$-critical substring is type 1F or 3F, ending in the second row.  Then the second row ends with $x\ge 1$ in the $2,3$-walk, and so after the $2'$ that appears in the top row the walk has $x\ge 2$.  The same argument as above now applies as well to show $F_2(F_1(T))=F_1(F_2(T))$ in this case.
  
  For the converse, suppose $F_2(F_1(T))=F_1(F_2(T))$ (and $F_1'(T)=\varnothing$).  If $\Delta=(1,1)$, the proof of Proposition \ref{prop:tableau-half-primed-square} shows that the top row of $T$ has the form $1^j2'3^k$ where the $2'$ is the final $F_2$-critical substring of type 4F.  Since $F_1'(T)=\varnothing$ this $2'$ is not simultaneously a $2$ and starts at $y\ge 1$ in the walk (and $x=0$ since it is type 4F).  But then $F_1(T)$ has a type 5F $F_2$-critical substring at the new $2$ that is inserted into the walk, a contradiction.
  
  Finally, assume for contradiction that $\Delta=(0,0)$.  By Lemma \ref{lem:Delta-tableaux}, the final $F_2$-critical substring starts in the second row and the $x$-coordinate of the walk changes upon inserting a $2$ in the top row.
  
  If the final $F_2$-critical substring is type 1F ending in the top row, then it is of the form $22'3'$ where the $2$ is the last entry in the second row, the $2'$ is in the top row, and the $3'$ is just after it in the top row.  Note that applying $F_1$ has the effect of inserting a $2$ to make this sequence $22'23'$ instead, which does move the endpoint of the walk to the right, and makes the new later $23'$ the final type 1F critical substring in the top row.  So $F_2(F_1(T))$ has the same second and third row as $T$, but $F_1(F_2(T))$ does not since the 1F critical substring starts in the second row.  Thus $F_2(F_1(T))\neq F_1(F_2(T))$ in this case.
  
  If the final $F_2$-critical substring is type 2F (which must end in the top row), then by the reasoning above in the $\Delta=(0,1)$ case, the position of the walk after the last $2$ in $T$ must be at $x=2$ and $y=0$, or at $x=2$ and $y=1$ with $3'$ as the next entry.  But then applying $F_1$ results in a later 1F or 3F $F_2$-critical substring in the top row due to the inserted $2$, and so $F_2(F_1(T))$ does not change the second row while $F_1(F_2(T))$ does.  Thus again they are not equal.  A similar argument applies if the final $F_2$-critical substring is type 1F or 3F contained in the second row.
  
  It follows that unless $\Delta=(0,1)$ or $(1,0)$, we have $F_1(F_2(T))\neq F_2(F_1(T))$.
\end{proof}

\subsection{Axiom \ref{ax:half-primed-stembridge}: Half-primed Octagon}

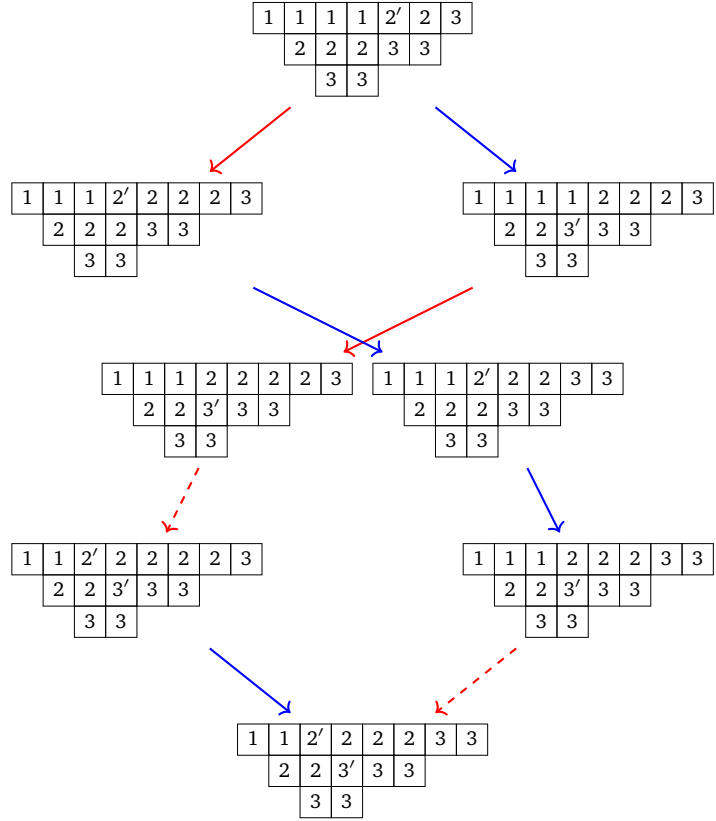
\begin{figure}[t]
 \begin{center}
\begin{tikzpicture}[x=0.6cm,y=0.6cm]
  \node (top) at (0,0) {\scriptsize $\begin{ytableau}  1 & 1 & 1 & 1 & 2' & 2 & 3 \\ \none & 2 & 2 & 2 & 3 & 3 \\ \none & \none & 3 & 3\end{ytableau}$};
  \node (left) at (-5,-4) {\scriptsize $\begin{ytableau}  1 & 1 & 1 & 2' & 2 & 2 & 2 & 3 \\ \none & 2 & 2 & 2 & 3 & 3 \\ \none & \none & 3 & 3\end{ytableau}$};
  \node (right) at (5,-4) {\scriptsize $\begin{ytableau} 1 & 1 & 1 & 1 & 2 & 2 & 2 & 3 \\ \none & 2 & 2 & 3' & 3 & 3 \\ \none & \none & 3 & 3\end{ytableau}$}; 
  \node (mid1) at (-3,-8) {\scriptsize $\begin{ytableau} 1 & 1 & 1 & 2 & 2 & 2 & 2 & 3 \\ \none & 2 & 2 & 3' & 3 & 3 \\ \none & \none & 3 & 3\end{ytableau}$}; 
  \node (mid2) at (3,-8) {\scriptsize $\begin{ytableau} 1 & 1 & 1 & 2' & 2 & 2 & 3 & 3 \\ \none & 2 & 2 & 2 & 3 & 3 \\ \none & \none & 3 & 3\end{ytableau}$};
  \node (left2) at (-5,-12) {\scriptsize $\begin{ytableau} 1 & 1 & 2' & 2 & 2 & 2 & 2 & 3 \\ \none & 2 & 2 & 3' & 3 & 3 \\ \none & \none & 3 & 3\end{ytableau}$};
  \node (right2) at (5,-12) {\scriptsize $\begin{ytableau} 1 & 1 & 1 & 2 & 2 & 2 & 3 & 3 \\ \none & 2 & 2 & 3' & 3 & 3 \\ \none & \none & 3 & 3\end{ytableau}$};
  \node (bot) at (0,-16) {\scriptsize $\begin{ytableau} 1 & 1 & 2' & 2 & 2 & 2 & 3 & 3 \\ \none & 2 & 2 & 3' & 3 & 3 \\ \none & \none & 3 & 3\end{ytableau}$};
  \draw [red, ->, thick] (top)--(left);
  \draw [blue, ->, thick] (top)--(right);
  \draw [red, ->, thick] (right)--(mid1);
  \draw [blue, ->, thick] (left)--(mid2);
  \draw [red, dashed, ->, thick] (mid1)--(left2);
  \draw [blue, ->, thick] (mid2)--(right2);
  \draw [red, dashed, ->, thick] (right2)--(bot);
  \draw [blue, ->, thick] (left2)--(bot);
  
\end{tikzpicture}
\end{center}
\caption{\label{fig:half-primed-stembridge}A half-primed octagon diagram.}
 \end{figure}

\begin{proposition}\label{prop:tableau-half-primed-stembridge}
  Suppose both $F_1(T)$ and $F_2(T)$ are defined and $F_1'(T)=\varnothing$.  Then 
  $$F_2(F_1'(F_1(F_2(T))))=F_1'(F_2(F_2(F_1(T))))\neq \varnothing \text{ and } F_1(F_2(T))\neq F_2(F_1(T))$$ if and only if $\Delta=(0,0)$ and $\vepshat_1(T)-\vepshat_1(f_2(T))=-1$.
\end{proposition}

\begin{proof}
First assume $\Delta=(0,0)$ and $\vepshat_1(T)-\vepshat_1(F_2(T))=-1$.  By Lemma \ref{lem:Delta-tableaux}, the final $F_2$-critical substring starts in the second row, and the $x$-coordinate of the endpoint of the $2,3$-walk changes after inserting a $2$ in the top row.

Observe that $F_1'$ is defined at $F_2(T)$. If the $1$-string through $F_2(T)$ is separated, this is because
\[\veps_1'(F_2(T))
= \veps_1(F_2(T)) - \vepshat_1(F_2(T))
= \veps_1(T) - (\vepshat_1(T) + 1)
= \veps_1'(T)-1 = 0,\]
and so $\varphi_1'(F_2(T)) = 1$. If the string is collapsed, we instead observe $\varphi_1(F_2(T)) \geq 2$ by Axiom \ref{ax:length} (since $\varphi_1(T) \geq 1$).
%
Thus there is a $2'$ in the top row of $T$ (since $F_1'(T) = \varnothing$), but not $F_2(T)$. 
 
 With one exception (below), the final $F$-critical substring of $T$ therefore has type 2F and begins at the rightmost $2$ on the second row and ends at the $2'$ in the first.  In order for the walk to be at the right position (so that the final $x$-coordinate changes upon inserting a $2$), and for there to be no later $F_2$-critical strings, $T$ must have the form
 $$\begin{ytableau} 1 & 1 & 1 & 1 & 1 & 2' & 2 & 2 & 3 \\ \none & 2 & 2 & 2 & 3 & 3 \\ \none & \none & 3 & 3\end{ytableau}$$
with reading word $3^a2^{a+1}3^b1^c2'2^b3^d$ or $3^a2^{a+1}3^b1^c2'2^{b-1}3'3^d$, where $a,b$ are not both $0$.
Then we have the resulting operations in Figure \ref{fig:half-primed-stembridge}, and we're done.

In the exceptional case where the critical string has type 1F, the entire second row must consist only of a single $2$ (and the $2'$ in the first row becomes a $2$ in the canonical form of $F_2(T)$). Then $T$ has two rows and reading word $21^c2'3^d$ (in fact the case $a=b=0$ above), and again the desired relation holds with no square (and in fact with all edges collapsed except the top $T \to F_1(T)$ edge and the bottom $F_2$ edge).

For the converse, suppose a half-primed octagon diagram holds and no square.  Since no square holds and $F_1'(T)$ is not defined, we can't have $\Delta=(1,0)$ or $(0,1)$ by Axiom \ref{ax:square} (Proposition \ref{prop:tableau-square}).  Suppose for contradiction that $\Delta=(1,1)$.  Then there is a half-primed square at $T$ by Axiom \ref{ax:half-primed-square} (Proposition \ref{prop:tableau-half-primed-square}), and the top row of $T$ is of the form $1^k2'3^j$ with the $2'$ as the final type 4F critical substring by the proof of Proposition \ref{prop:tableau-half-primed-square}.  But then after applying $F_1$ the top row looks like $1^{k-1}2'23^j$ and the $2$ in this row is now a type 5F critical substring for $F_2$, a contradiction to the half-primed octagon existing, unless the $2'$ was the first entry of the $2,3$ subword.  But the latter only happens if it is a one-row shape in which case $F_1'(T)$ exists, a contradiction.  So we get a contradiction and it follows that $\Delta=(0,0)$.
 
 Now, we know $F_1'(T)$ is undefined, and since $F_1$ and $F_1'$ commute and $F_1'(F_1(F_2(T)))$ is defined, it follows that $F_1'(F_2(T))$ is defined.  Thus the $\vepshat$ condition holds as well, since $F_1'$ is defined at $F_2(T)$ and not at $T$.
\end{proof}

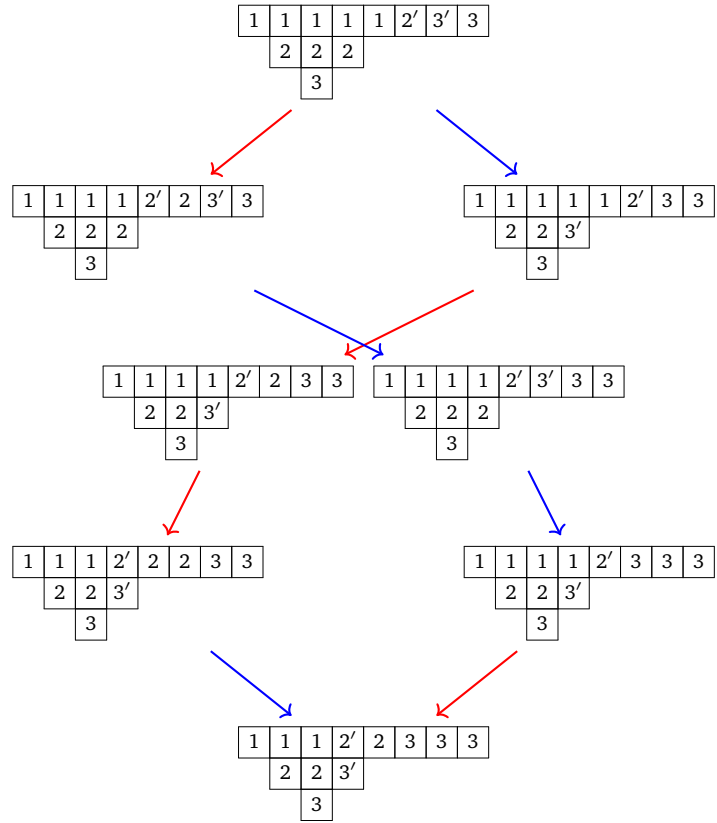
\begin{figure}[h!]
 \begin{center}
\begin{tikzpicture}[x=0.6cm,y=0.6cm]
  \node (top) at (0,0) {\scriptsize $\begin{ytableau}  1 & 1 & 1 & 1 & 1 & 2' & 3' & 3 \\ \none & 2 & 2 & 2 \\ \none & \none & 3 \end{ytableau}$};
  \node (left) at (-5,-4) {\scriptsize $\begin{ytableau}  1 & 1 & 1 & 1 & 2' & 2 & 3' & 3 \\ \none & 2 & 2 & 2 \\ \none & \none & 3 \end{ytableau}$};
  \node (right) at (5,-4) {\scriptsize $\begin{ytableau}  1 & 1 & 1 & 1 & 1 & 2' & 3 & 3 \\ \none & 2 & 2 & 3' \\ \none & \none & 3 \end{ytableau}$}; 
  \node (mid1) at (-3,-8) {\scriptsize $\begin{ytableau}  1 & 1 & 1 & 1 & 2' & 2 & 3 & 3 \\ \none & 2 & 2 & 3' \\ \none & \none & 3 \end{ytableau}$}; 
  \node (mid2) at (3,-8) {\scriptsize $\begin{ytableau}  1 & 1 & 1 & 1 & 2' & 3' & 3 & 3 \\ \none & 2 & 2 & 2 \\ \none & \none & 3 \end{ytableau}$};
  \node (left2) at (-5,-12) {\scriptsize $\begin{ytableau}  1 & 1 & 1 & 2' & 2 & 2 & 3 & 3 \\ \none & 2 & 2 & 3' \\ \none & \none & 3 \end{ytableau}$};
  \node (right2) at (5,-12) {\scriptsize $\begin{ytableau}  1 & 1 & 1 & 1 & 2' & 3 & 3 & 3 \\ \none & 2 & 2 & 3' \\ \none & \none & 3 \end{ytableau}$};
  \node (bot) at (0,-16) {\scriptsize $\begin{ytableau}  1 & 1 & 1 & 2' & 2 & 3 & 3 & 3 \\ \none & 2 & 2 & 3' \\ \none & \none & 3 \end{ytableau}$};
  \draw [red, ->, thick] (top)--(left);
  \draw [blue, ->, thick] (top)--(right);
  \draw [red, ->, thick] (right)--(mid1);
  \draw [blue, ->, thick] (left)--(mid2);
  \draw [red, ->, thick] (mid1)--(left2);
  \draw [blue, ->, thick] (mid2)--(right2);
  \draw [red, ->, thick] (right2)--(bot);
  \draw [blue, ->, thick] (left2)--(bot);
  
\end{tikzpicture}
\end{center}
\caption{\label{fig:Stembridge-1F-split}  An octagon diagram starting with a separated 1F critical $F_2$-substring.}
 \end{figure}

\subsection{Axiom \ref{ax:stembridge}: Octagon}

\begin{figure}[h!]
 \begin{center}
\begin{tikzpicture}[x=0.6cm,y=0.6cm]
  \node (top) at (0,0) {\scriptsize $\begin{ytableau}  1 & 1 & 1 & 1 & 2' & 2 & 3' & 3 \\ \none & 2 & 2 & 2 & 3' & 3 \\ \none & \none & 3 \end{ytableau}$};
  \node (left) at (-5,-4) {\scriptsize $\begin{ytableau}  1 & 1 & 1 & 2' & 2 & 2 & 3' & 3 \\ \none & 2 & 2 & 2 & 3' & 3 \\ \none & \none & 3 \end{ytableau}$};
  \node (right) at (5,-4) {\scriptsize $\begin{ytableau} 1 & 1 & 1 & 1 & 2' & 2 & 3' & 3 \\ \none & 2 & 2 & 3' & 3 & 3 \\ \none & \none & 3 \end{ytableau}$}; 
  \node (mid1) at (-3,-8) {\scriptsize $\begin{ytableau}  1 & 1 & 1 & 2' & 2 & 2 & 3' & 3 \\ \none & 2 & 2 & 3' & 3 & 3 \\ \none & \none & 3 \end{ytableau}$}; 
  \node (mid2) at (3,-8) {\scriptsize $\begin{ytableau} 1 & 1 & 1 & 2' & 2 & 3' & 3 & 3 \\ \none & 2 & 2 & 2 & 3' & 3 \\ \none & \none & 3 \end{ytableau}$};
  \node (left2) at (-5,-12) {\scriptsize $\begin{ytableau}  1 & 1 & 2' & 2 & 2 & 2 & 3' & 3 \\ \none & 2 & 2 & 3' & 3 & 3 \\ \none & \none & 3 \end{ytableau}$};
  \node (right2) at (5,-12) {\scriptsize $\begin{ytableau}  1 & 1 & 1 & 2' & 2 & 3' & 3 & 3 \\ \none & 2 & 2 & 3' & 3 & 3 \\ \none & \none & 3 \end{ytableau}$};
  \node (bot) at (0,-16) {\scriptsize $\begin{ytableau} 1 & 1 & 2' & 2 & 2 & 3' & 3 & 3 \\ \none & 2 & 2 & 3' & 3 & 3 \\ \none & \none & 3 \end{ytableau}$};
  \draw [red, ->, thick] (top)--(left);
  \draw [blue, ->, thick] (top)--(right);
  \draw [red, ->, thick] (right)--(mid1);
  \draw [blue, ->, thick] (left)--(mid2);
  \draw [red, ->, thick] (mid1)--(left2);
  \draw [blue, ->, thick] (mid2)--(right2);
  \draw [red, ->, thick] (right2)--(bot);
  \draw [blue, ->, thick] (left2)--(bot);
  
\end{tikzpicture}
\end{center}
\caption{\label{fig:stembridge-1F-whole}  An octagon diagram starting with a 1F critical $F_2$-substring fully contained in the second row.  The same type of figure results if it were a 3F critical substring instead.}
 \end{figure}
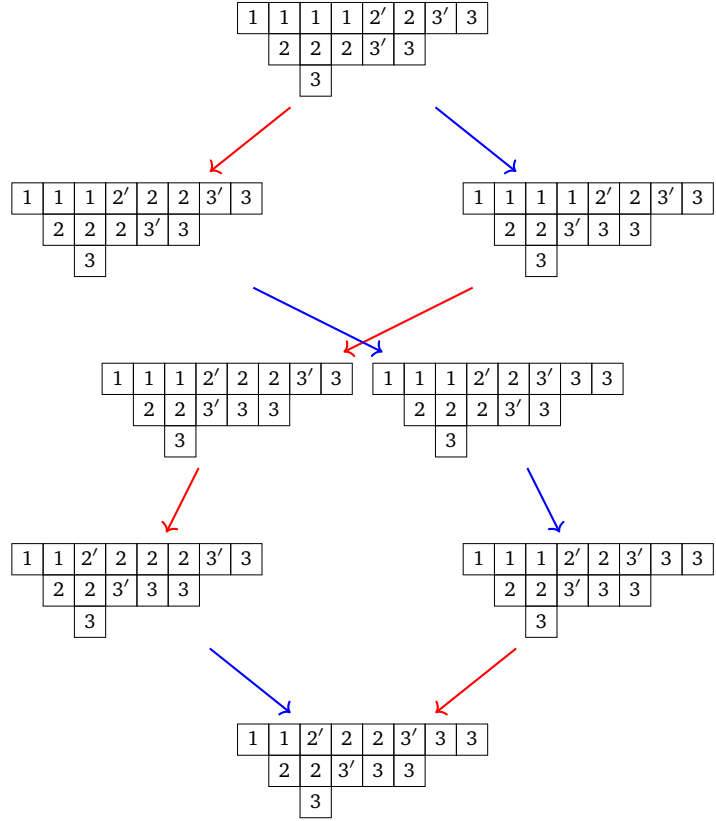

\begin{proposition}
Suppose $F_1(T)$ and $F_2(T)$ are both defined, and $F_1'(T)=\varnothing$.  We have that $$F_2(F_1^2(F_2(T)))=F_1(F_2^2(F_1(T)))\neq \varnothing\hspace{0.3cm}\text{ and }\hspace{0.3cm}F_1(F_2(T))\neq F_2(F_1(T))$$ if and only if $\Delta=(0,0)$ and $\widehat{\varphi}_1(F_2(T))\ge 2$.  
\end{proposition}

\begin{proof}
 Let $T$ have the specified properties and assume $\Delta=(0,0)$ and $\widehat{\varphi}_1(F_2(T))\ge 2$.  By Lemma \ref{lem:Delta-tableaux}, the final $F_2$-critical substring starts in the second row, and the $x$-coordinate of the $2,3$-walk changes after applying $F_1$.  As before, this change amounts to inserting a $2$ just before the top row's run of the form $2^\ast [3']3^\ast$, so the walk just before the last $2$ or $3'$ in this run is at $y=1$.  Note also that there is a $2'$ in the top row of $T$ since $F_1'(T)=\varnothing$. 
 
 If the final $F_2$-critical substring is type 2F, the analysis is the same as in Proposition \ref{prop:tableau-half-primed-stembridge} with the exception that there must be enough $1$s in the top row so that $F_1^2(F_2(T))$ is defined, for the $\widehat{\varphi}_1$ condition to hold.  It is easy to see that Figure \ref{fig:half-primed-stembridge} can be modified to an octagon diagram in this case.
 
 If the final $F_2$-critical substring is type 1F, suppose it ends in the top row, at a $3'$, where the second row has only $2$s, the last of which ends on the $x$-axis.  Then the diagram looks like that of Figure \ref{fig:Stembridge-1F-split}, and so the octagon rule is satisfied with no square.
 
 Otherwise, the final $F_2$-critical substring is type 1F or 3F completely contained in the second row.  In this case the diagram in Figure \ref{fig:stembridge-1F-whole} shows the octagon rule is satisfied with no square.
 
 For the converse, assume that $$F_2(F_1^2(F_2(T)))=F_1(F_2^2(F_1(T)))\neq \varnothing\hspace{0.3cm}\text{ and }\hspace{0.3cm}F_1(F_2(T))\neq F_2(F_1(T))$$ (and $F_1'(T)=\varnothing$).  Then the same proof as in Proposition \ref{prop:tableau-half-primed-stembridge} above shows that $\Delta=(0,0)$, and the relation itself indicates $\widehat{\varphi}_1(F_2(T))\ge 2$.  This completes the proof.
\end{proof}


\begin{thebibliography}{99}
  \bibitem{Assaf} S.~Assaf, E.~Kantarci Oguz, Crystal graphs for shifted tableaux, \textit{S\'{e}minaire Lotharingien de Combinatoire} (to appear as part of FPSAC 2018 conference proceedings).
  \bibitem{Berenstein} A.~Berenstein and D.~Kazhdan, Geometric and unipotent crystals, \textit{Geom. Funct. Anal.}, Special Volume, Part I (2000), 188--236.
  \bibitem{Berenstein2} A.~Berenstein and D.~Kazhdan, Geometric and Unipotent Crystals II: From Unipotent Bicrystals to Crystal Bases (with D. Kazhdan), \textit{Contemp. Math.}, \textbf{433}, Amer. Math. Soc., Providence, RI, 2007, 13--88.
  \bibitem{Schilling} D.~Bump and A.~Schilling, \textit{Crystal Bases: Representations and Combinatorics}, World Scientific (2017). 
  \bibitem{ChoiKwon} S.~Choi, J.~Kwon, Crystals and Schur P-positive expansions, preprint, 2017, arxiv:1707.02513.
  \bibitem{GJKKK} D.~Grantcharov, J.~Jung, S.~Kang, M.~Kashiwara, and M.~Kim, Crystal bases for the quantum queer superalgebra and semistandard decomposition tableaux, \textit{Transactions of the American Mathematical Society}, Vol.~366, No.~1, Jan.~2014, 457--489.
  \bibitem{GJKKK2} D.~Grantcharov, J.~Jung, S.~Kang, M.~Kashiwara, and M.~Kim, Crystal bases for the quantum queer superalgebra, \textit{J. Eur. Math. Soc.}, 17(7):1593--1627, 2015. 
  \bibitem{GLP} M.~Gillespie, J.~Levinson, and K.~Purbhoo, A crystal-like structure on shifted tableaux, submitted to \textit{J. Comb. Theory, Ser. A}, arxiv:1706.09969.
  \bibitem{GLP-geometry} M.~Gillespie, J.~Levinson, and K.~Purbhoo, Schubert curves in the orthogonal Grassmannian $\mathrm{OG}(n,2n+1)$, preprint.
  \bibitem{Hiroshima} T.~Hiroshima, $\mathfrak{q}$-crystal structure on primed tableaux and on signed unimodal factorizations of reduced words of type B, preprint, 2018 (arxiv:1803.05775).
  \bibitem{Kashiwara} M.~Kashiwara (1990), Crystallizing the $q$-analogue of universal enveloping algebras, \textit{Communications in Mathematical Physics}, \textbf{133} (2), pp.~249--260.
  \bibitem{Macdonald} I.~Macdonald, \textit{Symmetric Functions and Hall Polynomials}, Oxford Univ.\ Press, Oxford (1979).
  \bibitem{Morse-Schilling} J.~Morse and A.~Schilling, Crystal approach to affine Schubert calculus, \textit{Int Math Res Notices} (2016) (8): pp.~2239--2294.
  \bibitem{Sagan}  B.~Sagan,  Shifted tableaux, Schur Q-functions and a conjecture of R.\ Stanley, \textit{J.\ Combin.\ Theory Ser.\ A}, 45 (1987), pp.\ 62--103.
  \bibitem{sage} \emph{SageMath, the Sage Mathematics Software System (Version 7.6)}, The Sage Developers, 2017, {\tt http://www.sagemath.org}.
  \bibitem{Schur}  I.~Schur, \:{U}ber die Darstellung der symmetrischen und der alternierenden Gruppe durch gebrochene lineare Substitutionen, \textit{J. Reine Angew. Math.} \textbf{139} (1911), pp.~155--250.
  \bibitem{Schutzenberger} M.P. Sch\"{u}tzenberger, ``La correspondance de Robinson'' D. Foata (ed.), Combinatoire et Repr\'{e}sentation du Groupe Sym\'{e}trique, Lecture Notes in Mathematics, 579, Springer (1976) pp.\ 59--113.
  \bibitem{Stembridge} J.~Stembridge, A local characterization of simply-laced crystals, \textit{Trans.~Amer.~Math.~Soc.}, Vol.~355, No.~12, pp.~4807--4823.
\bibitem{Worley} D.~Worley, A theory of shifted Young tableau, \textit{Ph.D. Thesis}, MIT (1984).

\end{thebibliography}
\end{document}